\documentclass[11pt,a4paper,draft]{article}
\usepackage{mathrsfs}

\usepackage[leqno]{amsmath}
\usepackage{amssymb,amsthm,upref,amscd}
\usepackage{arydshln}
\usepackage[T1]{fontenc}
\numberwithin{equation}{section}

\theoremstyle{plain}
\newtheorem{Thm}{Theorem}[section]
\newtheorem{Lem}[Thm]{Lemma}
\newtheorem{Cor}[Thm]{Corollary}
\newtheorem{Prop}[Thm]{Proposition}
\newtheorem*{Thm*}{Theorem}
\theoremstyle{definition}
\newtheorem{Def}[Thm]{Definition}
\newtheorem{Rem}[Thm]{Remark}

\newcommand{\equ}{equation}
\newcommand{\C}{\mathbb{C}}

\newcommand{\R}{\mathbb{R}}
\newcommand{\Z}{\mathbb{Z}}

 \DeclareMathOperator{\dist}{dist}

 \DeclareMathOperator{\sgn}{sgn}

\let\nhatoksa=\theenumi
\let\nhatoksb=\labelenumi
\let\nhatoksc=\theenumii
\let\nhatoksd=\labelenumii
\newlength{\nhalengtha}
\newlength{\nhalengthb}
\newlength{\nhalengthc}
\newcommand{\resetenum}{
\let\theenumi=\nhatoksa
\let\labelenumi=\nhatoksb
\let\theenumii=\nhatoksc
\let\labelenumii=\nhatoksd
\setlength{\leftmargini}{\nhalengtha}
\setlength{\leftmarginii}{\nhalengthb}
\setlength{\labelwidth}{\nhalengthc}
}

\newcommand\rr{\mathcal{R}}

\def\msa{\mathscr{A}}

\def\msj{\mathscr{J}}
\def\msk{\mathscr{K}}
\def\msl{\mathscr{L}}
\def\msm{\mathscr{M}}
\def\msn{\mathscr{N}}

\def\msr{\mathscr{R}}

\def\mst{\mathscr{T}}

\def\ov{\overline}

\def\pa {\partial}

\def\op{\oplus}

\def\De{\Delta}

\def\ka {\kappa}
\def\al {\alpha}
\def\bt {\beta}
\def\de {\delta}
\def\Ga {\Gamma}
\def\ga {\gamma}
\def\lm {\lambda}
\def\Lam {\Lambda}

\def\Om {\Omega}
\def\sa {\sigma}
\def\vr {\varepsilon}
\def\va {\varphi}

\def\sgn{\hbox{sgn}}

\newcommand{\bkt}[1]{\left(#1\right)}

\newcommand{\inp}[2]{\left\langle#1,#2\right\rangle}

\title{Localized concentration of semi-classical states for
nonlinear Dirac equations}

\author{Yanheng Ding\footnote{Email address: dingyh@math.ac.cn}, \,
Tian Xu\footnote{Corresponding author, email address: xutian@amss.ac.cn}   \and {\small
Institute
of Mathematics, Academy of Mathematics and Systems Science,} \\
{\small  Chinese Academy of Sciences, 100190 Beijing, China} }

\date{}

\begin{document}
\maketitle

\begin{abstract}
The present paper studies concentration phenomena of semiclassical
approximation of a massive Dirac equation with general nonlinear
self-coupling:
\[
-i\hbar\al\cdot\nabla w+a\bt w+V(x)w=g(|w|)w \,.
\]
Compared with some existing issues, the most interesting results
obtained here are twofold: the solutions concentrating around local
minima of the external potential; and the nonlinearities assumed to
be either super-linear or asymptotically linear at the
infinity. As a consequence one sees that, if there are $k$ bounded
domains $\Lam_j\subset\R^3$ such that $-a<\min_{\Lam_j}
V=V(x_j)<\min_{\pa\Lam_j}V$, $x_j\in\Lam_j$, then the
$k$-families of solutions $w_\hbar^j$ concentrates around $x_j$
as $\hbar\to 0$, respectively. The proof relies on variational
arguments: the solutions are found as critical points of an energy
functional. The Dirac operator has a continuous spectrum which is
not bounded from below and above, hence the energy functional is
strongly indefinite. A penalization technique is developed here to
obtain the desired solutions.

\vspace{.5cm}


\vspace{.5cm}


\end{abstract}

\section{Introduction and main result}

This paper is motivated by some works appeared in recent
years concerning the nonlinear Dirac equation
\begin{\equ}\label{nd}
-i\hbar\pa_t\psi = ic\hbar\sum_{k=1}^3\al_k\pa_k\psi
- mc^2\bt \psi - M(x)\psi + f(x,|\psi|)\psi
\end{\equ}
for $(t,x)\in\R\times\R^3$. Here $c$ is the speed of light, $\hbar$
is Planck's constant, $m$ is
the mass of the electron and
$\alpha_1$, $\alpha_2$, $\alpha_3$ and $\beta$ are $4\times4$
complex Pauli matrices:
\[
\beta=\left(
\begin{array}{cc}
I&0\\
0&-I
\end{array}\right), \quad \alpha_k=\left(
\begin{array}{cc}
0&\sigma_k\\
\sigma_k&0
\end{array}\right), \quad k=1,2,3,
\]
with
\[
\sigma_1=\left(
\begin{array}{cc}
0&1\\
1&0
\end{array}\right), \quad \sigma_2=\left(
\begin{array}{cc}
0&-i\\
i&0
\end{array}\right), \quad \sigma_3=\left(
\begin{array}{cc}
1&0\\
0&-1
\end{array}\right) \,.
\]
The external fields $M(x)$ and $f(x,|\psi|)$ in \eqref{nd}
arise in mathematical models of particle physics,
especially in nonlinear topics. They are inspired by approximate
descriptions of the external forces involve only functions of fields.
The nonlinear self-coupling $f(x,|\psi|)$, which describes
a self-interaction in Quantum electrodynamics,
gives a closer description of many particles found
in the real world.  Typical examples can be found
in the self-interacting scalar theories, where
the nonlinear function $f$ can be both polynomial
and nonpolynomial (that includes the cases: $|\psi|^\lm$, $\sin|\psi|$ etc.).
Various nonlinearities are considered to be
possible basis models for unified field theories (see \cite{FLR},
\cite{FFK}, \cite{Iva} etc. and references therein).

Owning to the representation of \eqref{nd}, we are interested in solutions
of the form $\psi(t,x)=\exp(i\xi t/\hbar)\va(x)$ which are called the
standing waves (stationary states). It is easily checked
a $\psi$ of this form satisfies equation \eqref{nd} if and only if
$\va$ solves
\begin{\equ}\label{nd2}
-ic\hbar\sum_{k=1}^3\al_k\pa_k\va
+ mc^2\bt \va + \big( M(x)+\xi \big)\va =  f(x,|\va|)\va \,.
\end{\equ}
Dividing \eqref{nd2} by $c$, we are led to study the equation
of the form
\begin{\equ}\label{nd3}
-i\hbar\sum_{k=1}^3\al_k\pa_k\va
+ mc\bt \va + V(x)\va =  g(x,|\va|)\va \,.
\end{\equ}
Initiated by \cite{Ding2010}, the case $V(x)\equiv0$ and
$g(x,|\va|)=P(x)|\va|^{p-2}$ for $p\in(2,3)$ with a given external
potential $P(x)$ is considered and a global variational technique is
developed to show the least energy solution exists provided that
$\hbar$ is sufficiently small and the solution concentrates around
the maxima point of $P$ as $\hbar\to0$. This method and result were
later generalized jointly with co-authors in \cite{Ding-Lee-Ruf,
Ding2012, Ding Ruf} to competing potentials and critical
nonlinearities, that is
\begin{\equ}\label{lp1}
V(x)\not\equiv0 \,, \quad   \min_{x\in\R^3}V(x)<
\liminf_{|x|\to\infty}V(x) \,,
\end{\equ}
and $g(x,|\va|)\sim P(x)(|\va|^{p-2}+|\va|)$ for $p\in(2,3)$. As was
shown in \cite{Ding-Lee-Ruf, Ding2012, Ding Ruf}, the semi-classical
solutions concentrate around some certain points that depend on both
linear and nonlinear potentials. Further investigations on the
existence of solutions concentrating at certain points to nonlinear
Dirac equations (including Maxwell-Dirac systems and
Klein-Gordon-Dirac systems) under different conditions have also
appeared in \cite{Tian1, Tian2, Tian3}.

The method mentioned above basically depends on the global condition
of the external potentials (see for example \eqref{lp1}).
An interesting question, which  motivates
the present work, is whether one can find solutions which concentrate around
local minima (or maxima) of a external potential.

As we will see, the answer is (or at least partially) affirmative.
For  small $\hbar$, the solitary waves are
referred to as semi-classical states. The physical interpretation of such
states is the following. One of the basic principles of quantum mechanics
is the correspondence principle, according to which when $\hbar\to0$, the laws of
quantum mechanics must reduce to those of classical mechanics.
To describe the transition
from quantum to classical mechanics, the existence of solutions
$\va_\hbar$, $\hbar$ small,
possesses an important physical interest. In the present paper,
denoted by $\vr=\hbar$,
$\al=(\al_1, \al_2, \al_3)$ and
$\al\cdot\nabla=\sum_{k=1}^3\al_k\pa_k$, we are concerned
with the following stationary nonlinear Dirac equation
\begin{\equ}\label{D2}
-i\vr\al\cdot\nabla w + a\bt w + V(x) w = g(|w|)w \,,
\text{ for } w\in H^1(\R^3,\C^4)
\end{\equ}
where $a=mc>0$ is constant.

Let us assume the external linear potential satisfies
\begin{itemize}
\item[$(V_1)$] {\it $V$ is locally H\"older continuous and $\max|V|<a$}.
\end{itemize}
And for the nonlinear fields, by writing $G(s):=\int^{s}_0g(t)t
\,dt$, we begin with the sup-linear case:
\begin{itemize}
\item[$(g_1)$] {\it $g(0)=0$, $g\in C^1(0,\infty)$, $g'(s)>0$};
\item[$(g_2)$]
\begin{itemize}
\item[$(i)$] {\it there exist $p\in (2, 3)$, $c_1>0$ such that $g(s)\leq
c_1(1+s^{p-2})$ for $s\geq 0$};
\item[$(ii)$] {\it  there exists $\theta>2$ such that
$0< G(s)\leq \frac{1}{\theta}g(s)s^2$ for all $s>0$};
\end{itemize}
\item[$(g_3)$] {\it the function $s\mapsto g'(s)s+g(s)$ is nondecreasing}.
\end{itemize}

Our first result is as follows
\begin{Thm}\label{main theorem}
Suppose $(V_1)$ and $(g_1)$-$(g_3)$ are satisfied. Assume additionally that
there is a bounded domain $\Lam$ in $\R^3$ such that
\begin{\equ}\label{local}
\underline c:=\min_\Lam V < \min_{\pa\Lam} V \, .
\end{\equ}
Then for all $\vr>0$ small,
\begin{itemize}
\item[$(i)$] a solution $w_\vr\in\cap_{q\geq2}W^{1,q}(\R^3,\C^4)$
to \eqref{D2} exists;

\item[$(ii)$] $|w_\vr|$ possesses a (global)
maximum point $x_\vr$ in $\Lam$ such that
\[
\lim_{\vr\to 0}V(x_\vr)= \underline c \, ,
\]
and
\[
|w_\vr(x)|\leq C \exp{\big( -\frac{c}{\vr}|x-x_\vr|\big)}
\]
for certain constants $C,c>0$;

\item[$(iii)$] setting $v_\vr(x)=w_\vr(\vr x + x_\vr)$, as
$\vr\to0$, we have
$v_\vr$ converges in $H^1$ to a least energy solution of
\[
-i\al\cdot\nabla v + a\bt v + \underline c \, v = g(|v|)v \,.
\]
\end{itemize}
\end{Thm}

We mention that if \eqref{lp1}
is assumed, then \eqref{local} is satisfied naturally. So
Theorem \ref{main theorem} can be seen as a local
(or even a stronger) version
of the consequences proved in \cite{Ding2012}.

Our next result is concerned with the asymptotically linear
case. Denoted by $\widehat G(s):=\frac12 g(s)s^2-G(s)$,
the hypothesis $(g_2)$ will be replaced by
\begin{itemize}
\item[$(g_2')$]
\begin{itemize}
\item[$(i)$]{\it there exists $b>\max|V|+a$ such that $g(s)\to b$ as
$s\to\infty$ };
\item[$(ii)$] {\it $\widehat G(s)>0$ if $s>0$, and
$\widehat G(s)\to\infty$ as $s\to\infty$}.
\end{itemize}
\end{itemize}

Our result reads as

\begin{Thm}\label{main theorem 2}
Suppose $(V_1)$, $(g_1)$, $(g_2')$ and $(g_3)$
are satisfied. Let \eqref{local} be satisfied for
some bounded domain $\Lam\subset\R^3$.
Then for all $\vr>0$ small,
\begin{itemize}
\item[$(i)$] a solution $w_\vr\in\cap_{q\geq2}W^{1,q}(\R^3,\C^4)$
to \eqref{D2} exists;

\item[$(ii)$] $|w_\vr|$ possesses a (global)
maximum point $x_\vr$ in $\Lam$ such that
\[
\lim_{\vr\to 0}V(x_\vr)= \underline{c} \, ,
\]
and
\[
|w_\vr(x)|\leq C \exp{\big( -\frac{c}{\vr}|x-x_\vr|\big)}
\]
for certain constants $C,c>0$;

\item[$(iii)$] setting $v_\vr(x)=w_\vr(\vr x + x_\vr)$, as
$\vr\to0$, we have
$v_\vr$ converges in $H^1$ to a least energy solution of
\[
-i\al\cdot\nabla v + a\bt v + \underline{c}\,v = g(|v|)v \,.
\]
\end{itemize}
\end{Thm}

To our knowledge Theorem \ref{main theorem 2}
is the first concentrating result concerned
with asymptotically linear
Dirac equation. Moreover, as mentioned before,
our assumption \eqref{local} is rather weak:
no restriction on the global behavior of $V$
is required other than $(V_1)$. In particular, the behavior
of $V$ outside $\Lam$ is irrelevant. On the other hand,
hypotheses $(g_1)$-$(g_3)$ are satisfied by a large class of
nonlinearities, which may appear in the self-interacting scalar theories,
including:
\begin{itemize}
\item[$1$.] $G(s)=s^p$ with $p\in(2,3)$ for the super-linear case;
\item[$2$.] $G(s)=\frac b2 s^2\Big(
1-\frac1{\ln(e+s)} \Big)$ for the asymptotically linear case.
\end{itemize}
Due to the above observations, we have an immediate consequence of
our main theorems:

\begin{Cor}\label{corollary}
Suppose $(V_1)$, $(g_1)$, $(g_3)$ and either $(g_2)$
or $(g_2')$ are satisfied. If there exist mutually disjoint
bounded domains $\Lam_j$ $j=1,\dots,k$
and constants $c_1<c_2<\cdots<c_k$ such that
\begin{\equ}\label{multi-domain}
c_j:=\min_{x\in\Lam_j}V(x)<\min_{x\in\pa\Lam_j}V(x) \,.
\end{\equ}
Then for all $\vr>0$ small,
\begin{itemize}
\item[$(i)$] there exist at least $k$ solutions $w_\vr^j\in\cap_{q\geq2}W^{1,q}(\R^3,\C^4)$
to \eqref{D2}, $j=1,\dots,k$;

\item[$(ii)$] $|w_\vr^j|$ possesses a (global)
maximum point $x_\vr^j$ in $\Lam_j$ such that
\[
\lim_{\vr\to 0}V(x_\vr^j)= c_j \, ,
\]
and
\[
|w_\vr^j(x)|\leq C \exp{\big( -\frac{c}{\vr}|x-x_\vr^j|\big)}
\]
for certain constants $C,c>0$;

\item[$(iii)$] setting $v_\vr^j(x)=w_\vr^j(\vr x + x_\vr^j)$, as
$\vr\to0$, we have
$v_\vr^j$ converges in $H^1$ to a least energy solution of
\[
-i\al\cdot\nabla v + a\bt v + c_j v = g(|v|)v \,.
\]
\end{itemize}
\end{Cor}

We remark here that in Corollary \ref{corollary} the solutions
can be separated provided $\vr$ is small since $\Lam_j$
are mutually disjoint. Furthermore, if $c_1$ in
\eqref{multi-domain} is a global minimum
of $V$, then  Corollary \ref{corollary} describes
a multiple concentrating phenomenon generalizing the results in
\cite{Ding2012}.

It is standard that \eqref{D2} is equivalent to, by letting $u(x)=
w(\vr x)$,
\begin{\equ}\label{equivalent D1}
-i \al\cdot\nabla u + a\bt u + V_\vr(x) u = g(|u|)u
\end{\equ}
where $V_\vr (x)=V(\vr x)$. We will in the sequel focus on this
equivalent problem.
The proofs of Theorem \ref{main theorem} and Theorem
\ref{main theorem 2} are variational,
and rely upon a reduction argument and
a linking structure of the energy functional
associate to \eqref{equivalent D1}.
Since the functional is strongly indefinite,
we have to recover a compactness condition at some minimax level
when $\vr$ is small.
All we do is to build a modification of the energy functional.
In such a way, the functional is proved to satisfy the
so-called Cerami compactness
condition. And then, for $\vr$ sufficiently small,
a solution associate
to the linking level is indeed a solution to the original
equation. The modification of the functional corresponds
to a penalization technique "outside $\Lam$", and this is
why no other global assumptions are required for $V$.

There have been enormous investigations on existence
and concentration phenomenon of semi-classical
states of nonlinear Schr\"odinger
equation arising in the {\it non-relativistic} quantum mechanics:
\begin{\equ}\label{schrodinger}
\hbar^2\De w - V(x)w + f(w)=0 \quad w\in H^1(\R^n) \,.
\end{\equ}
It is the first time, Floer and Weinstein, in \cite{Floer},
proved in the one dimensional case and for $f(w)=w^3$
that a single spike solution concentrating around
any given non-degenerate critical point
of the potential $V(x)$. Oh \cite{O1, O2} extended
this result in higher dimension and
for $f(u)=|u|^{p-1}u$ $(1<p<N+2/N-2)$.
The arguments in \cite{Floer, O1, O2}
are based on a Lyapunov-Schmidt reduction
and rely on the uniqueness and non-degeneracy
of the ground state solutions of the
autonomous problems:
\[
-\De v + V(x_0) v = f(v) \quad v\in H^1(\R^n) \quad
(x_0\in \R^n) \,.
\]
Subsequently, variational methods were found suitable to such
issues and the existence of spike layer solutions in the semi-classical
limit has been established under various conditions of $V(x)$.
Particularly, initiated by Rabinowitz \cite{Rab}, the
existence of positive solutions of \eqref{schrodinger}
for small $\hbar>0$ is proved whenever
\[
\liminf_{|x|\to\infty}V(x)> \inf_{x\in\R^n}V(x) \,.
\]
And these solutions concentrate around the global
minimum points of $V$ when $\hbar\to0$,
as was shown by Wang \cite{Wang X}.
It should be pointed out that
M. Del Pino and P. Felmer in \cite{Del Pino} firstly
succeeded in proving a localized version of
the concentration behaviour of semi-classical
solutions. In \cite{Del Pino}, assuming $\inf V=V_0>0$ and
\eqref{local}
for some bounded domain $\Lam$, the authors showed the
existence of a single-peak solution which concentrates around the
minimum points of $V$ in $\Lam$. Their
approach depends on a penalization argument and Mountain-pass
theorem.
Note that, since the Schr\"odinger operator $-\De+V$
is bounded from below, techniques based on the Mountain-pass
theorem are well applied to the investigation.
For further related results, we refer the readers to
\cite{Ambrosetti1, Ambrosetti2, Byeon-JJ,
Byeon-Wang, Del Pino2, Jeanjean-Tanaka} and their references,
moreover, for certain results on solutions
with multiple spike patterns for Hamiltonian elliptic systems
we refer to \cite{MH}.

It is quite natural to ask if certain similar results can be
obtained for nonlinear Dirac equations arising in the {\it
relativistic} quantum mechanics. Mathematically, the problems in
Dirac equations are difficult because they are strongly indefinite
in the sense that both the negative and positive parts of the
spectrum of Dirac operator are unbounded and consist of essential
spectrums.  To this end, instead of the "Mountain-pass structure",
we need a deep insight into the linking structure of strongly
indefinite functional (\cite{Rab0}).
 To illustrate this point, we will see in
Section \ref{auxiliary results} how local properties of $V$ lead to
an essential linking structure of the strongly indefinite energy
functional.

An outline of this paper is as follows: Section \ref{TVF}
is  devoted to define the modification of the functional needed for
the proof of our main results, and prove some
preliminary results. In Section \ref{auxiliary results},
we prove some auxiliary results and linking structure.
Theorem \ref{main theorem} and
Theorem \ref{main theorem 2} are proved in Section \ref{PMT}.

\section{The variational framework}\label{TVF}


In the sequel, by $|\cdot|_q$ we denote the usual $L^q$-norm, and
$(\cdot,\cdot)_2$ the usual $L^2$-inner product. Let
$H_0=-i\al\cdot\nabla+a\bt$ denote the self-adjoint operator on
$L^2\equiv L^2(\R^3,\C^4)$ with domain
$\mathcal{D}(H_0)=H^1\equiv H^1(\R^3,\C^4)$. It is
well known that $\sa(H_0)=\sa_c(H_0)=\R\setminus(-a,a)$
where $\sa(\cdot)$ and $\sa_c(\cdot)$ denote the spectrum and
the continuous spectrum. Thus the space $L^2$ possesses the
orthogonal decomposition:
\begin{equation}\label{l2dec}
L^2=L^+\oplus L^-,\ \ \ \ u=u^++u^-
\end{equation}
so that $H_0$ is positive definite (resp. negative definite) in
$L^+$ (resp. $L^-$). Let $E:=\mathcal{D}(|H_0|^{1/2})=H^{1/2}$
be equipped with the inner product
$$
\inp{u}{v}=\Re(|H_0|^{1/2}u,|H_0|^{1/2}v)_2
$$
and the induced norm $\|u\|=\inp{u}{u}^{1/2}$, where $|H_0|$
and $|H_0|^{1/2}$ denote respectively the absolute value of
$H_0$ and the square root of $|H_0|$. Since
$\sa(H_0)=\R\setminus(-a,a)$, one has
\begin{equation}\label{l2ineq}
a|u|_2^2\leq\|u\|^2 \quad  \mathrm{for\ all\ }u\in E.
\end{equation}
Note that this norm is equivalent to the usual $H^{1/2}$-norm, hence
$E$ embeds continuously into $L^q$ for all $q\in[2,3]$ and compactly
into $L_{loc}^q$ for all $q\in[1,3)$.

It is clear that $E$ possesses the following decomposition
\begin{equation}\label{Edec}
E=E^+\oplus E^-\ \ \mathrm{with\ \ }E^{\pm}=E\cap L^{\pm},
\end{equation}
orthogonal with respect to both $(\cdot,\cdot)_2$ and
$\inp{\cdot}{\cdot}$ inner products. And remarkably, this decomposition
of $E$ induces also a natural decomposition of $L^q$ for every $q\in(1,+\infty)$:
\begin{Prop}\label{lpdec}
Let $E^+\op E^-$ be the decomposition of $E$ according to the positive
and negative part of $\sa(H_0)$. Then, set $E^\pm_q:=E^\pm\cap L^q$ for $q\in(1,\infty)$,
there holds
\[
L^q={\rm cl}_q\, E^+_q\op {\rm cl}_q\, E^-_q
\]
with ${\rm cl}_q$ denoting the closure in $L^q$. More precisely, there exists $d_q>0$
for every $q\in(1,\infty)$ such that
\[
d_q|u^\pm|_q\leq |u|_q \quad {\rm for\ all\ } u\in E\cap L^q.
\]
\end{Prop}

In $L^q$'s (for $q\neq2$), by $\op$ we mean the topologically direct sum.
Before proving Proposition \ref{lpdec} we would like to introduce the following
definition for {\it Multipliers} (see \cite[Chapter 4]{Stein}) which plays an important role in our arguments.

\begin{Def}
Let $m$ be a bounded measurable function on $\R^n$, we associate a linear operator
$T_m$ on $L^2\cap L^q$ by $(T_m u){\hat\ }(\xi)=m(\xi)\hat u(\xi)$ where $\hat u$
denotes the Fourier transform of $u$. We say that $m$ is a multiplier for $L^q$
($1\leq q\leq\infty$) if whenever $u\in L^2\cap L^q$ then $T_m u \in L^q$ (notice
it is automatically in $L^2$), and $T_m$ is bounded, that is,
\begin{\equ}\label{def1}
|T_m u|_q\leq A\cdot |u|_q, \quad  u\in L^2\cap L^q  \text{ (with } A
\text{ independent of } u \text{)}.
\end{\equ}
\end{Def}

Observe that if \eqref{def1} is satisfied, and $p<\infty$, then $T_m$ has a unique bounded
extension to $L^q$, which again satisfies the same inequality.

\begin{proof}[Proof of Proposition \ref{lpdec}]
First we remark that in this context, the spatial domain is $\R^3$.
Now recall the definitions for the matrices $\sa_k$, $k=1,2,3$ (see Pauli matrices),
and let us study $H_0=-i\al\cdot\nabla+a\bt$. It is a differential operator with
constant coefficients. In the Fourier domain $\xi=(\xi_1,\xi_2,\xi_3)$, it becomes
the operator of multiplication by the matrix
\[
\hat H_0(\xi)=\left(
\begin{array}{cc}
0& \sum_{k=1}^3\xi_k\sa_k \\
\sum_{k=1}^3\xi_k\sa_k &0
\end{array} \right) + \left(
\begin{array}{cc}
a&0\\
0&-a
\end{array} \right).
\]
By classical calculus we have that $\hat H_0(\xi)$ has two eigenvalues: $\pm\sqrt{a^2+|\xi|^2}$.
Now, denote  $P^\pm$ the projections on $E$ with kernel $E^\mp$. We see that in the Fourier domain,
$P^\pm$ are multiplication operators by bounded smooth matrix-valued functions of $\xi$:
\[
(P^+u){\hat\ }(\xi)=\Big(\frac12+\frac{a}{2\sqrt{a^2+|\xi|^2}}\Big)
\left(
\begin{array}{cc}
I
&  \Sigma(\xi) 
\\[2pt]
 \Sigma(\xi)  
&  A(\xi) 
\end{array} \right)
\left(\begin{array}{c}
\hat U\\[2pt]
\hat V
\end{array} \right)
\]
\[
(P^-u){\hat\ }(\xi)=\Big(\frac12+\frac{a}{2\sqrt{a^2+|\xi|^2}}\Big)
\left(
\begin{array}{cc}
A(\xi)& -\Sigma(\xi)
 \\[2pt]
-\Sigma(\xi) & I
\end{array} \right)
\left(\begin{array}{c}
\hat U\\[2pt]
\hat V
\end{array} \right)
\]
with $I$ being the $2\times2$ identity matrix and
\[
A(\xi)=\frac{\sqrt{a^2+|\xi|^2}-a}{a+\sqrt{a^2+|\xi|^2}}\cdot I, \quad
\Sigma(\xi)=\sum_{k=1}^3\frac{\xi_k\sa_k}{a+\sqrt{a^2+|\xi|^2}}.
\]
Here we have used the notation $\hat u=(\hat U,\hat V)\in\C^4$ and
$\hat U=(\hat u_1,\hat u_2)\in \C^2$, $\hat V=(\hat u_3,\hat u_4)\in\C^2$.

In order that $P^\pm$ are multipliers for $L^q$, we need to use the
{\it Marcinkiewicz multiplier theorem} on $\R^3$ (see \cite[Chapter 4, Theorem 6']{Stein}).
A direct calculation shows that, for each component, the absolute value
of all $k$-th ($0< k\leq3$) order partial derivatives
for the multiplication functions are bounded by $B/|\xi|^k$ for some constant $B>0$.
And hence, as an immediate consequence,
 $P^\pm$ are {\it multipliers} for $L^q$ for all $q\in(1,\infty)$.
This implies that $P^\pm$ are continuous with respect to the $L^q$-norms. By
noting that $P^\pm(E^\mp)=\{0\}$, one easily sees that $P^\pm$ extend to continuous
projections on $L^q$ (still denoted by $P^\pm$) with $P^\pm({\rm cl}_q E^\mp_q)=\{0\}$.
And this completes the proof.
\end{proof}

\begin{Rem}
It is of great importance for the projections from $H^{1/2}:=E=E^+\op E^-$
onto $E^+$ (or $E^-$) to be continuous in the $L^q$'s and not only in $H^{1/2}$.
This is not the case for every direct sum in $H^{1/2}$. In fact, the proof of
Proposition \ref{lpdec} implies on the splitting of $L^q$'s that:
For every $q\in (1,\infty)$, $L^q$ can be split into topologically direct sum
of two (infinite dimensional) subspaces
which, accordingly, are the positive and negative projected spaces of the Dirac operator
$H_0$.
\end{Rem}


To introduce the variational formulation of problem \eqref{equivalent D1},
we define the "energy" functional
\begin{\equ}\label{X0}
\aligned
\Phi_\vr(u)&=\frac12\int_{\R^3}H_0 u \cdot \bar u + \frac12 \int_{\R^3} V_\vr(x)|u|^2
-\int_{\R^3}G(|u|)\\
&=\frac12\big(
\|u^+\|^2-\|u^-\|^2 \big) +\frac12 \int_{\R^3} V_\vr(x)|u|^2
-\Psi(u)
\endaligned
\end{\equ}
for $u=u^+ + u^-\in E$. 
Standard arguments show that, under our assumptions, $\Phi_\vr\in C^2(E,\R)$ and any
critical point of $\Phi_\vr$ is a (weak) solution to \eqref{equivalent D1}.

Next, we introduce a modification of \eqref{X0},
afterwards, we will prove the modified functional
satisfies the $(C)_c$ condition.
Choose $\xi>0$ be the value at which
$g'(\xi)\xi+g(\xi)=\frac{a-|V|_\infty}{2}$. Let us consider
$\tilde{g}\in C^1(0,\infty)$ such that
\[
\frac{d}{ds}\big(
\tilde g (s)s \big)=\left\{
\aligned
& g'(s)s+g(s) \quad {\rm if\ } s<\xi \\
& \frac{a-|V|_\infty}{2} \quad {\rm if\ } s>\xi
\endaligned \right. \, ,
\]
and define
\begin{\equ}\label{mod}
f(\cdot,s)=\chi_\Lam g(s) + (1-\chi_\Lam)\tilde g(s) \, ,
\end{\equ}
where $\Lam$ is a bounded domain as in the assumptions of
Theorem \ref{main theorem} and $\chi_\Lam$ denotes its
characteristic function. One should keep in mind here that
$\Lam$ has to be rescaled when we consider the modified
rescaled equation \eqref{equivalent D1}.

It is standard to check that
$(g_1)$ and $(g_3)$ implies that $f$ is a Caratheodory function
and it satisfies
\begin{itemize}
\item[$(f_1)$] {\it $f_s(x,s)$ exists everywhere, $f(x,s)s=o(s)$
uniformly in $x$ as $s\to0$};
\item[$(f_2)$] {\it $0\leq f(x,s)s\leq g(s)s$ for all $x$};
\item[$(f_3)$] {\it $0< 2F(x,s)\leq f(x,s)s^2\leq \frac{a-|V|_\infty}{2}s^2$ for all $x\not\in \Lam$ and $s>0$};
\item[$(f_4)$]
\begin{itemize}
\item[$(i)$] {\it if $(g_2)$ is satisfied, then
 $0< F(x,s)\leq \frac{1}{\theta}f(x,s)s^2$ for all $x\in\Lam$ and $s>0$},
\item[$(ii)$] {\it if $(g_2')$ is satisfied, then $\widehat F(x,s)>0$
if $s>0$};
\end{itemize}
\item[$(f_5)$] {\it $\frac{d}{ds}\big( f(x,s)s \big)\geq 0$ for all $x$
and $s>0$};
\item[$(f_6)$] {\it either $(g_2)$ or $(g_2')$ is satisfied,
$\widehat F(x,s)\to \infty$ as $s\to\infty$
uniformly in $x$}.
\end{itemize}
Here we used the notation $F(x,s)=\int_0^s f(x,t)t\, dt$ and
$\widehat F(x,s)=\frac12 f(x,s)s^2-F(x,s)$. Now, we define
the modified functional $\widetilde\Phi_\vr:E\to\R$ as
\[
\widetilde\Phi_\vr(u)=\frac12\big(
\|u^+\|^2-\|u^-\|^2 \big) +\frac12 \int_{\R^3} V_\vr(x)|u|^2
-\Psi_\vr(u) \,,
\]
where $\Psi_\vr(u)=\int_{\R^3} F(\vr x,|u|)$. Then, we see that
$\widetilde\Phi_\vr\in C^2(E,\R)$.

We show next $\widetilde\Phi_\vr$ satisfies the compactness condition.
In virtue of $(f_4)(i)$, we have
\begin{\equ}\label{tilde F}
\widehat F(x,s)\geq \frac{\theta-2}{2\theta}f(x,s)s^2\geq
\frac{\theta-2}{2}F(x,s)>0
\end{\equ}
for all $x\in\Lam$ and $s>0$ provided $(g_2)$ is satisfied.
Recall that, by $(f_1)$ and $(f_2)$, there exist $r_1>0$ small
enough and $a_1>0$ such that
\begin{\equ}\label{s<r1}
f(x,s)\leq \frac{a-|V|_\infty}4 \quad \text{for all $s\leq r_1$, $\ x\in\R^3$}
\end{\equ}
and if $(g_2)$ is satisfied,
for $s\geq r_1$, $f(x,s)\leq a_1 s^{p-2}$, so
$\big(f(x,s)s\big)^{\sa_0-1}\leq a_2 s$ with
\[
\sa_0:=\frac{p}{p-1} \,,
\]
which, jointly with $(f_4)(i)$, yields (see \eqref{tilde F})
\begin{\equ}\label{g-sigma0 estimate}
\big(f(x,s)s\big)^{\sa_0}\leq a_2 f(x,s)s^2\leq a_3 \widehat{F}(x,s)
\quad \text{for all $s\geq r_1$ and $x\in\Lam$}.
\end{\equ}
\begin{Lem}\label{C-c condition}
For each $\vr>0$, let $\{u_n\}$ be a sequence such that
$\widetilde\Phi_\vr(u_n)$ is bounded and
$(1+\|u_n\|)\widetilde\Phi_\vr'(u_n)\to0$. Then $\{u_n\}$
has a convergent subsequence.
\end{Lem}
\begin{proof}
We first show that the sequence $\{u_n\}$ is bounded in $E$.
In fact, the
representation of $\widetilde\Phi_\vr$ implies that there is $C>0$ such that
\begin{\equ}\label{up bdd}
C\geq \widetilde\Phi_\vr(u_n)-\frac12\widetilde\Phi_\vr'(u_n)u_n
=\int \widehat F(\vr x,|u_n|)
\end{\equ}
and
\begin{\equ}\label{Cc-o1}
\aligned
o(1)=&\,\widetilde\Phi_\vr'(u_n)(u_n^+-u_n^-)\\
=&\,\|u_n\|^2
 +\Re\int V_\vr(x)u_n\cdot\ov{(u_n^+-u_n^-)} \\
&\,-\Re\int f(\vr x,|u_n|)u_n\cdot\ov{(u_n^+-u_n^-)} \,.
\endaligned
\end{\equ}

{\it Case 1.} $(f_4)(i)$ occurs.

By the definition of $f$ and \eqref{Cc-o1}, we have
\begin{\equ}\label{N1}
\aligned
&\,\|u_n\|^2-|V|_\infty \int |u_n|\cdot|u_n^+-u_n^-|\\
\leq& \int f(\vr x,|u_n|)|u_n|\cdot|u_n^+-u_n^-|+o(1) \\
\leq& \int_{\Lam_\vr} f(\vr x,|u_n|)|u_n|\cdot|u_n^+-u_n^-|
+\frac{a-|V|_\infty}2 \int |u_n|\cdot|u_n^+-u_n^-| \\
&\, +o(1) \,,
\endaligned
\end{\equ}
where $\Lam_\vr:=\{x\in\R^3:\vr x\in\Lam\}$. Thus, from
\eqref{s<r1} and \eqref{g-sigma0 estimate}, we easily check that
\[
\aligned
&\,\frac{a-|V|_\infty}{4a}\|u_n\|^2\\
\leq&
\int_{\{x\in\Lam_\vr:|u_n(x)|\geq r_1\}}
f(\vr x,|u_n|)|u_n|\cdot|u_n^+-u_n^-| + o(1) \\
\leq&\bigg( \int_{\{x\in\Lam_\vr:|u_n(x)|\geq r_1\}}
\big(f(\vr x,|u_n|)|u_n|\big)^{\sa_0} \bigg)^{1/\sa_0} \cdot
|u_n^+-u_n^-|_p\\
&\, + o(1) \,.
\endaligned
\]
It follows from \eqref{tilde F},
\eqref{up bdd} and  $E$ embeds
continuously into $L^p$, we find
\[
\frac{a-|V|_\infty}{4a}\|u_n\|^2\leq C_1 \|u_n\| + o(1) \,.
\]
Then $\{u_n\}$ is bounded in $E$ as desired.

{\it Case 2.} $(f_4)(ii)$ occurs.

Assume contrarily that $\|u_n\|\to\infty$ as $n\to\infty$ and set
$v_n=u_n/\|u_n\|$. Then $|v_n|_2^2\leq C_2$ and $|v_n|_3^2\leq C_3$.
It follows from \eqref{l2ineq} and \eqref{Cc-o1} that
\[
\aligned
o(1)=&\,\|u_n\|^2\bigg(
\|v_n\|^2 + \Re\int V_\vr(x)v_n\cdot\ov{(v_n^+-v_n^-)} \\
&\,-\Re\int f(\vr x,|u_n|)v_n\cdot\ov{(v_n^+-v_n^-)} \, \bigg)  \\
\geq&\,\|u_n\|^2\bigg(
\frac{a-|V|_\infty}{a}-
\Re\int f(\vr x,|u_n|)v_n\cdot\ov{(v_n^+-v_n^-)} \, \bigg) \,.
\endaligned
\]
Thus
\begin{\equ}\label{infequ1}
\liminf_{n\to\infty} \,\Re\int f(\vr x,|u_n|)v_n\cdot\ov{(v_n^+-v_n^-)}
\geq \ell:=\frac{a-|V|_\infty}{a} \,.
\end{\equ}
To get a contradiction, let us first set
\[
d(r):=\inf\big\{
\widehat F(\vr x, s): \, x\in\R^3, \text{ and } s>r \big\} \,,
\]
\[
\Om_n(\rho,r):=\big\{ x\in\R^3: \, \rho\leq|u_n(x)|<r \big\} \,,
\]
and
\[
c_\rho^r:=\inf\bigg\{
\frac{\widehat F(\vr x,s)}{s^2}: \, x\in\R^3, \text{ and }
\rho\leq s\leq r \bigg\} \,.
\]
By $(f_6)$, $d(r)\to\infty$ as $r\to\infty$ and by definition
\[
\widehat F\big( \vr x, u_n(x) \big) \geq c_\rho^r|u_n(x)|^2 \quad
\text{for all } x\in\Om_n(\rho,r) \,.
\]
From \eqref{up bdd}, we have
\[
C\geq \int_{\Om_n(0,\rho)}\widehat F\big(\vr x, u_n(x)\big)
+c_\rho^r\int_{\Om_n(\rho,r)}|u_n|^2 + d(r)|\Om_n(r,\infty)| \,.
\]
Observe that $|\Om_n(r,\infty)|\leq C/d(r)\to0$ as $r\to\infty$
unformly in $n$, and, for any fixed $0<\rho<r$,
\[
\int_{\Om_n(\rho,r)}|v_n|^2=\frac1{\|u_n\|^2}
\int_{\Om_n(\rho,r)}|u_n|^2\leq \frac{C}{c_\rho^r\|u_n\|^2}\to0
\]
as $n\to\infty$.

Now let us choose $0<\de<\ell/3$. By $(f_1)$ there is $\rho_\de>0$
such that $|f(\vr x, s)|<\frac{\de}{C_2}$ for all $s\in[0,\rho_\de]$.
Consequently,
\[
\int_{\Om_n(0,\rho)} |f(\vr x, |u_n|)| \cdot |v_n| \cdot
|v_n^+-v_n^-| \leq \frac{\de}{C_2} |v_n|_2^2\leq \de
\]
for all $n$. Recall that, by $(g_1)$, $(g_2')$ and \eqref{mod},
$f(\vr x, s)\leq b$ for all $(x,s)$. Using H\"{o}lder inequality
we can take $r_\de$ large so that
\[
\aligned
&\int_{\Om_n(r_\de,\infty)}|f(\vr x, |u_n|)| \cdot |v_n| \cdot
|v_n^+-v_n^-|\\
\leq&\, b\int_{\Om_n(r_\de,\infty)} |v_n| \cdot
|v_n^+-v_n^-| \\
\leq&\, b \cdot |\Om_n(r_\de,\infty)|^{1/6} \cdot
 |v_n|_2 \cdot |v_n^+-v_n^-|_3  \\
\leq&\, C_b|\Om_n(r_\de,\infty)|^{1/6}
\leq \de
\endaligned
\]
for all $n$. Moreover, there is $n_0$ such that
\[
\aligned
&\int_{\Om_n(\rho_\de,r_\de)}|f(\vr x, |u_n|)| \cdot |v_n| \cdot
|v_n^+-v_n^-|  \\
\leq&\, b \int_{\Om_n(\rho_\de,r_\de)} |v_n| \cdot
|v_n^+-v_n^-|  \\
\leq&\, b\cdot |v_n|_2 \Big(
\int_{\Om_n(\rho_\de,r_\de)} |v_n|^2 \Big)^{1/2}
\leq \de
\endaligned
\]
for all $n\geq n_0$. Therefore, for $n\geq n_0$,
\[
\int f(\vr x, |u_n|)|v_n||v_n^+-v_n^-| \leq 3\de<\ell \,,
\]
which contradicts \eqref{infequ1}.

To prove the compactness, we recall some
direct observations: since $E$ embeds compactly
into $L_{loc}^q$ for all $q\in[1,3)$, the boundedness of $\{u_n\}$
implies that there exists $u\in E$ satisfying (after passing to a
subsequence if necessary)
\[
\text{$u_n\rightharpoonup u$ in $E$, \quad $u_n\to u$ in $L_{loc}^q$}
\]
for $q\in[1,3)$. Set $z_n=u_n-u$, we remark that $\{z_n\}$ is bounded
and $z_n\rightharpoonup 0$ in $E$ and $z_n\to0$ in $L_{loc}^q$ as
$n\to\infty$ for $q\in[1,3)$.

We conclude from $\{u_n\}$ is a bounded $(P.S.)$ sequence that
\begin{\equ}\label{L1}
o(1)=\inp{u_n^+}{z_n^+}+\Re\int V_\vr(x)u_n\cdot \ov{z_n^+}
-\Re\int f(\vr x,|u_n|)u_n\cdot \ov{z_n^+} \,,
\end{\equ}
\begin{\equ}\label{L2}
0=\inp{u^+}{z_n^+}+\Re\int V_\vr(x)u\cdot \ov{z_n^+}
-\Re\int f(\vr x,|u|)u\cdot \ov{z_n^+} \,,
\end{\equ}
\begin{\equ}\label{L3}
o(1)=-\inp{u_n^-}{z_n^-}+\Re\int V_\vr(x)u_n\cdot \ov{z_n^-}
-\Re\int f(\vr x,|u_n|)u_n\cdot \ov{z_n^-} \,,
\end{\equ}
\begin{\equ}\label{L4}
0=-\inp{u^-}{z_n^-}+\Re\int V_\vr(x)u\cdot \ov{z_n^-}
-\Re\int f(\vr x,|u|)u\cdot \ov{z_n^-} \,.
\end{\equ}
On the other hand, we find
\begin{\equ}\label{L5}
\left\{
\aligned
&\Re\int f(\vr x,|u|)u\cdot \ov{z_n^+}=o(1) \,,\\
&\Re\int f(\vr x,|u|)u\cdot \ov{z_n^-}=o(1) \,,\\
&\Re\int f(\vr x,|u_n|)u\cdot\ov{(z_n^+-z_n^-)}=o(1) \,.
\endaligned \right.
\end{\equ}
Therefore, again with the definition of $f$,
we deduce from \eqref{L1}-\eqref{L5} that
\[
\frac{a-|V|_\infty}{4a}\|z_n\|^2\leq
\Re\int_{\Lam_\vr}f(\vr x,|u_n|)z_n\cdot\ov{(z_n^+-z_n^-)}+o(1) \,.
\]
Since $\Lam_\vr$ is bounded for any fixed $\vr$, we have $\|z_n\|=o(1)$
as $n\to\infty$, which completes the proof.
\end{proof}

The pervious lemma makes it possible to use the critical point theory
to find critical points of $\widetilde\Phi_\vr$. We will formulate an
appropriate minimax level for $\widetilde\Phi_\vr$.

First set, for $r>0$, $B_r=\{u\in E:\|u\|\leq r\}$, and for
$e\in E^+\setminus\{0\}$
\[
E_e:=E^-\oplus\R^+ e
\]
with $\R^+=[0,+\infty)$. It follows from $(g_1)$ and $(f_2)$ that there
exists $C>0$ such that
\begin{\equ}\label{N2}
F(x,s)\leq \frac{a-|V|_\infty}{4}s^2+C s^p \quad \text{for all } s\geq0 \,.
\end{\equ}
So we have:
\begin{Lem}\label{link1}
There are $r>0$ and $\tau>0$, both independent of $\vr$, such that
$\widetilde\Phi_\vr|_{B_r^+}\geq0$ and $\widetilde\Phi_\vr|_{S_r^+}\geq\tau$,
where
\[
B_r^+=B_r\cap E^+=\{u\in E^+:\|u\|\leq r\},
\]
\[
S_r^+=\pa B_r^+=\{u\in E^+:\|u\|= r\}.
\]
\end{Lem}
\begin{proof}
Recall that $|u|_p^p\leq C_p\|u\|^p$ for all $u\in E$ by
Sobolev's embedding theorem. The conclusion follows easily because,
for $u\in E^+$
\[
\aligned
\widetilde\Phi_\vr (u)&=\frac12 \|u\|^2-\frac{1}{2}\int V_\vr(x)|u|^2
 -\Psi_\vr (u)\\
 &\geq\frac12 \|u\|^2-\frac{|V|_\infty}{2}|u|_2^2
 -\bkt{\frac{a-|V|_\infty}{4}|u|_2^2+C|u|_p^p}\\
 &\geq\frac{a-|V|_\infty}{4a}\|u\|^2-C'\|u\|^p
\endaligned
\]
with $C,C'>0$ independent of $u$ and $p>2$ (see \eqref{N2}).
\end{proof}

Let $\msk_\vr:=\big\{ u\in E\setminus\{0\}:\, \widetilde\Phi_\vr'(u)=0 \big\}$
be the critical set of $\widetilde\Phi_\vr$.
By virtue of Lemma \ref{C-c condition}, using the same iterative
argument of \cite[Proposition 3.2]{Sere2} and \cite[Lemma 3.19]{Tian1},
we obtain the following

\begin{Lem}\label{critical points in W1,q}
If $u\in\msk_\vr$ with $|\widetilde\Phi_\vr(u)|\leq C$. Then,
for any $q\geq2$, $u\in W^{1,q}(\R^3,\C^4)$ with
$\|u\|_{W^{1,q}}\leq C_q$, where $C_q$ depends only on $C$ and $q$;
\end{Lem}

\begin{Rem}\label{uniformly estimate}
Let $\msl_\vr$ be the set of all least energy solutions of
$\widetilde\Phi_\vr$.
Let $u\in\msl_\vr$, then $|\widetilde\Phi_\vr(u)|\leq C$ for
some constant $C>0$ (this will be proved in
Lemma \ref{c-vr leq}). Therefore, as a consequence of Lemma
\ref{critical points in W1,q}, we see (together with the Sobolev
embedding theorem) that there exist $C_\infty>0$ independent of $\vr$
such that
\[
|u|_\infty\leq C_\infty \quad \text{for all } u\in\msl_\vr \,.
\]
\end{Rem}

\section{Some auxiliary results}\label{auxiliary results}

Firstly, it is easily checked that,
for any $x_0\in\R^3$, setting
$\tilde V_\vr(x)= V\big( \vr(x+x_0) \big)$,
if $\tilde u$ is a solution of
\[
-i\al\cdot\nabla \tilde u + a\bt \tilde u +
\tilde V_\vr(x) \tilde u = g(|\tilde u|) \tilde u
\]
then $u(x)=\tilde u(x-x_0)$ solves \eqref{equivalent D1}.
Thus, without loss of generality, we can
assume that $0\in\Lam$ and $V(0)=\min_\Lam V$.

\subsection{Functional reduction}
Let us mention here that a reduction of an strongly indefinite functional
to a functional on $E^+$ is well know under stronger differentiability conditions,
see for example \cite{Ackermann, MH, MY} and the papers concerned with Dirac operator.
In \cite{MH, MY} a reduction in two
steps has been performed: first to $E^+$ and then to a Nehari manifold on $E^+$.

Motivated by the above papers, we shall use the reduction approaches to find
critical points of $\Phi_\vr$. 
Suppose $(f_1)$-$(f_6)$ are satisfied,
for a fixed $u\in E^+$, let $\phi_u:E^-\to\R$
defined by $\phi_u(v)=\widetilde\Phi_\vr(u+v)$. We infer
\begin{\equ}\label{N3}
\phi_u(v) \leq \frac{a+|V|_\infty}{2a}\|u\|^2
-\frac{a-|V|_\infty}{2a}\|v\|^2 \,.
\end{\equ}
Moreover, we have, for any $v,w\in E^-$,
\begin{\equ}\label{N4}
\aligned
\phi_u''(v)[w,w]=&\,-\|w\|^2-\int V_\vr(x)|w|^2
 -\Psi_\vr''(u+v)[w,w]  \\
\leq&\, - \frac{a-|V|_\infty}a \|w\|^2 \,.
\endaligned
\end{\equ}
Indeed, a direct calculation shows that
\[
\aligned
\Psi_\vr''(u+v)[w,w]=&\,\int \Big[
f_s(\vr x,|u+v|)|u+v| \bigg(
\frac{(u+v)\cdot\ov w}{|u+v|\cdot|w|} \bigg)^2 \\
&\, + f(\vr x,|u+v|)  \Big] |w|^2 \,dx  \,.
\endaligned
\]
Then, it follows from $(f_5)$ and $\Big(
\frac{(u+v)\cdot\ov w}{|u+v|\cdot|w|} \Big)^2\leq1$ that
$\Psi_\vr''(u+v)[w,w]\geq0$. As a consequence of \eqref{N3}
and \eqref{N4} (that is anti-coercion and concavity),
there exists a unique $h_\vr:E^+\to E^-$ such
that
\[
\widetilde\Phi_\vr(u+h_\vr(u))=\max_{v\in E^-}\widetilde\Phi_\vr(u+v)\,.
\]
From the definition of $h_\vr$, we have
\begin{\equ}\label{identity5}
\aligned
0\leq&\,\widetilde\Phi_\vr\big( u+h_\vr(u) \big) -
\widetilde\Phi_\vr( u )  \\
=&\, -\frac12\|h_\vr(u)\|^2 +\frac12\int V_\vr(x)|u+h_\vr(u)|^2
- \Psi_\vr\big( u+h_\vr(u) \big) \\
&\, - \frac12\int V_\vr(x)|u|^2 + \Psi_\vr(u)  \\
\leq&\, -\frac{a-|V|_\infty}{2a}\|h_\vr(u)\|^2 + \frac{|V|_\infty}a\|u\|^2
+\Psi_\vr(u)
\endaligned
\end{\equ}
for all $u\in E^+$. Hence
the boundedness of $\Psi_\vr$ implies that of $h_\vr$. Set
$\pi: E^+\op E^- \to E^-$ by
\[
\pi(u,v)= P^-\circ\rr\circ\widetilde\Phi_\vr'(u+v) \,,
\]
where $P^-:E\to E^-$ is the projection and $\rr:E^*\to E$
denotes the isomorphism induced from
the Riesz representation theorem. Noting that for every
$u\in E^+$, we have
\begin{\equ}\label{N5}
\pi\big( u,h_\vr(u) \big)=0 \,.
\end{\equ}
Since $\pi_v(u,v)=P^-\circ\rr\circ
\widetilde\Phi_\vr''(u+v)\big|_{E^-}$,
from \eqref{N4} it follows that $\pi_v\big( u,h_\vr(u) \big)$
is an isomorphism with
\begin{\equ}\label{N6}
\big\| \pi_v\big( u,h_\vr(u) \big)^{-1}
\big\|\leq \frac{a}{a-|V|_\infty}
\end{\equ}
for every $u\in E^+$. So \eqref{N5} and \eqref{N6} together with the
implicit function theorem imply the uniquely defined map
$h_\vr: E^+\to E^-$ is $C^1$ smooth with
\begin{\equ}\label{N7}
h_\vr'(u)=-\pi_v\big( u,h_\vr(u) \big)^{-1}\circ
\pi_u\big( u,h_\vr(u) \big) \,,
\end{\equ}
where $\pi_u(u,v):=P^-\circ\rr\circ
\widetilde\Phi_\vr''(u+v)\big|_{E^+}$.
Define
\[
I_\vr: E^+\to\R \,, \quad  I_\vr(u)=\widetilde\Phi_\vr(u+h_\vr(u)) \,.
\]
We see directly that critical points of $I_\vr$ and $\widetilde\Phi_\vr$
are in one-to-one correspondence via the injective map
$u\mapsto u+h_\vr(u)$ from $E^+$ into $E$.

It is clear that for any $u\in E^+$ and $v\in E^-$, by setting
$z=v-h_\vr(u)$ and $l(t)=\widetilde\Phi_\vr(u+h_\vr(u)+tz)$, one
has $l(1)=\widetilde\Phi_\vr(u+v)$, $l(0)=\widetilde\Phi_\vr(u+h_\vr(u))$
and $l'(0)=0$. Hence we deduce $l(1)-l(0)=\int_0^1(1-s)l''(s)ds$. Consequently,
we have
\[
\aligned
&\widetilde\Phi_\vr(u+v)-\widetilde\Phi_\vr(u+h_\vr(u)) \\
=&\,\int_0^1(1-s)\widetilde\Phi_\vr''(u+h_\vr(u)+sz)[z,z] \,ds\\
=&\,-\int_0^1(1-s)\Big(
\|z\|^2+\int V_\vr(x)|z|^2\,dx \Big) \,ds \\
&\, -\int_0^1(1-s)\Psi_\vr''(u+h_\vr(u)+sz)[z,z] \,ds \,,
\endaligned
\]
which implies
\begin{\equ}\label{identity1}
\aligned
&\int_0^1(1-s)\Psi_\vr''(u+h_\vr(u)+sz)[z,z] \,ds\\
&+\frac12\|z\|^2+\frac12\int V_\vr(x)|z|^2=
\widetilde\Phi_\vr(u+h_\vr(u))-\widetilde\Phi_\vr(u+v) \,.
\endaligned
\end{\equ}

\subsection{The limit equation}\label{TLE}

For $\mu\in(-a,a)$, assume $(g_1)$,$(g_3)$
and either $(g_2)$ or $(g_2')$ are satisfied,
let us consider the equation
\begin{\equ}\label{limit equ1}
-i\al\cdot\nabla u + a\bt u + \mu u = g(|u|)u, \qquad
u\in H^1(\R^3,\C^4) \,.
\end{\equ}
Its solutions are critical points of the functional
\[
\mst_\mu(u):=\frac12\big(
\|u^+\|^2-\|u^-\|^2 \big) +\frac\mu2\int |u|^2 - \Psi(u)
\]
defined for $u=u^++u^-\in E=E^+\op E^-$. Denote the critical
set, the least energy and the set of least energy solutions of
$\mst_\mu$ as follows
\[
\aligned
&\msk_\mu:=\{u\in E:\ \mst_\mu'(u)=0\},\\
&\ga_\mu:=\inf\{\mst_\mu(u):\
 u\in\msk_\mu\setminus\{0\}\},\\
&\msr_\mu:=\{u\in\msk_\mu:\
 \mst_\mu(u)=\ga_\mu,\
 |u(0)|=|u|_\infty\}.
\endaligned
\]
And as before, we introduce the following notations:
\[
\aligned
&\msj_\mu:E^+\to E^-, \quad \mst_\mu\big(u+\msj_\mu(u)\big)
=\max_{v\in E^-}\mst_\mu(u+v) \,; \\
&J_\mu: E^+\to\R, \quad J_\mu=\mst_\mu\big(
u+\msj_\mu(u) \big) \,.
\endaligned
\]
Remark that, from the definition of $\msj_\mu$, we have
\begin{\equ}\label{identity6}
\mst_\mu'\big(u+\msj_\mu(u))z=0
\quad \text{for all } z\in E^- \,.
\end{\equ}
And, similar to \eqref{identity5}, there holds
\begin{\equ}\label{msj-estimate}
\|\msj_\mu(u)\|^2\leq \frac{2|\mu|}{a-|\mu|}\|u\|^2
+\frac{2a}{a-|\mu|}\Psi_\vr(u) \,.
\end{\equ}

\subsubsection{The super-linear case}

The following lemma is from \cite{Ding2} (see also \cite{Ding2008})

\begin{Lem}\label{ding2008}
For the equation \eqref{limit equ1} we have:
\begin{itemize}
\item[$(i)$] $\msk_\mu\setminus\{0\}\not=\emptyset$, $\ga_\mu>0$
 and $\msk_\mu\subset\cap_{q\geq2}W^{1,q}$;

\item[$(ii)$] $\ga_\mu$ is attained and $\msr_\mu$
 is compact in $H^1(\R^3,\C^4)$;

\item[$(iii)$] there exist $C,c>0$ such that
$$|w(x)|\leq C\exp (-c|x|)$$
for all $x\in\R^3$ and $w\in\msr_\mu$.
\end{itemize}
\end{Lem}

Observe that assumption $(g_2)(ii)$ implies that for any $\de>0$ there is a constant
$c_\de>0$ such that
\[
G(s)\geq c_\de s^\theta-\de s^2 \quad \text{for all } s\geq0.
\]
Directly, for $v\in E^-$, $u=te+v\in E_e$, we have that
\[
\aligned
\mst_\mu(u)&=\frac{t^2}2\|e\|^2-\frac{\|v\|^2}2+\frac\mu2\int|te+v|^2-\int G(|te+v|) \\
 &\leq \frac{a+|\mu+2\de|}{2a}t^2\|e\|^2-\frac{a-|\mu+2\de|}{2a}\|v\|^2-c_\de\int|te+v|^\theta.
\endaligned
\]
And hence, by Proposition \ref{lpdec},
\[
\mst_\mu(u)\leq \frac{a+|\mu+2\de|}{2a}t^2\|e\|^2
-\frac{a-|\mu+2\de|}{2a}\|v\|^2-C_{\de,\theta}t^\theta|e|_\theta^\theta
\quad \text{on } E_e.
\]
As a consequence of the above estimates we have the following lemma,
the specific proofs can be found in \cite{Ding2012} (see also \cite{Ding2010,
Ding2008})

\begin{Lem}\label{limit equ link}
There hold the following properties:
\begin{itemize}
\item[$(1)$] For any $e\in E^+\setminus\{0\}$, we have
$\mst_\mu(u)\to -\infty$ provided $u\in E_e$
and $\|u\|\to\infty$.

\item[$(2)$] Set $\Ga_\mu=\big\{
\nu\in C([0,1],E^+):\, \nu(0)=0,\ J_\mu(\nu(1))<0 \big\}$,
we have
\[
\ga_\mu=\inf_{\nu\in\Ga_\mu}\max_{t\in[0,1]}
J_\mu(\nu(t))=\inf_{u\in E^+\setminus\{0\}}\max_{t\geq0}
J_\mu(tu) \,.
\]

\item[$(3)$] If $\mu_1>\mu_2$, then $\ga_{\mu_1}>\ga_{\mu_2}$.
\end{itemize}
\end{Lem}

Notice that, similar to \eqref{identity1}, we have for $u\in E^+$,
$v\in E^-$ and $z=v-\msj_\mu(u)$
\begin{\equ}\label{identity2}
\aligned
&\int_0^1(1-s)\Psi''(u+\msj_\mu(u)+sz)[z,z] \,ds\\
&+\frac12\|z\|^2+\frac\mu2\int |z|^2=
\mst_\mu(u+\msj_\mu(u))-\mst_\mu(u+v) \,.
\endaligned
\end{\equ}

\subsubsection{The asymptotically linear case}

Firstly, let $(E_s)_{s\in\R}$ denote the spectral family of
$H_0:=-i\al\cdot\nabla+a\bt$. Choose a number
$a+|V|_\infty<\ka<b$. Since $H_0$ is invariant under the
action of $\Z^3$, the subspace $Y_0:=(E_\ka-E_0)L^2$ is
infinite-dimensional, and
\begin{\equ}\label{Y0-ineq}
a|u|_2^2\leq\|u\|^2\leq\ka|u|_2^2 \quad
\text{for all } u\in Y_0 \,.
\end{\equ}
Take an element $e\in Y_0$ arbitrarily, we have

\begin{Lem}\label{ding book lemma 7.7}
There holds the following:
\begin{itemize}
\item[$(1)$] $\sup\mst_\mu(E_e)<+\infty$, and
$\mst_\mu(u)\to -\infty$ provided $u\in E_e$
and $\|u\|\to\infty$.

\item[$(2)$] For any $u\in E^+\setminus\{0\}$, taking $t\to\infty$,
then either $J_\mu(tu)\to +\infty$ or $J_\mu(tu)\to -\infty$.
\end{itemize}
\end{Lem}
\begin{proof}
For the proof of (1), we refer \cite[Lemma 7.7]{Ding2}. To show $(2)$,
let us first assume $\sup_{t\geq0}J_\mu(tu)=M<+\infty$. Following
\eqref{identity6}, a direct calculation shows
\begin{\equ}\label{identity7}
\aligned
\frac{d}{dt} J_\mu(tu)& = \frac1t J_\mu'(tu)tu
= \frac1t \mst_\mu'\big(tu+\msj_\mu(tu)\big)\big(
tu+\msj_\mu'(tu)tu\big)  \\
&=\frac1t \mst_\mu'\big(tu+\msj_\mu(tu)\big)\big(
tu+\msj_\mu(tu)\big)   \\
&=\frac{2J_\mu(tu)}t-\frac2t \int \widehat G
\big(|tu+\msj_\mu(tu)|\big) \,.
\endaligned
\end{\equ}
For $r>0$, we infer
\begin{\equ}\label{estimate G}
\aligned
& \int \widehat G
\big(|tu+\msj_\mu(tu)|\big) \\
\geq&\,
\int_{\big\{ x\in\R^3:|u+\msj_\mu(tu)/t|\geq r \big\}}
\widehat G \big(|tu+\msj_\mu(tu)|\big)   \\
\geq&\, \widehat G(rt) \cdot \textit{meas}
\big\{ x\in\R^3:|u+\msj_\mu(tu)/t|\geq r \big\} \,.
\endaligned
\end{\equ}
Since the family $\{\msj_\mu(tu)/t\}_{t>0}\subset E^-$
is bounded (due to \eqref{msj-estimate} and $(g_2')(i)$),
we must have $\textit{meas}
\big\{ x\in\R^3:|u+\msj_\mu(tu)/t|\geq r \big\}\geq\bar\de$
with some $\bar\de>0$ for all $t>0$ provided $r>0$ is small.
Indeed, if such $\bar\de$ does not exist, we then have
\[
\msj_\mu(t_ju)/t_j \rightharpoonup -u  \quad
\text{in } E
\]
for some sequence $\{t_j\}$. However, this will imply $u=0$ since
$u\in E^+$, which is a contradiction. Now, from \eqref{estimate G}
and $(g_2')(ii)$, we deduce that
\[
\aligned
\frac{d}{dt}J_\mu(tu) & \leq
\frac{2J_\mu(tu)}t -2 \bar\de  \cdot
\frac{\widehat G(rt)}{t}  \\
&\leq \frac{2J_\mu(tu)}{t}-
\frac{3M}{t}  \\
&\leq -\frac{M}{t}
\endaligned
\]
for $t$ sufficiently large. Thus we have
$J_\mu(tu)=\int_0^t \frac{d}{dt}J_\mu(tu)\to-\infty$
as $t\to +\infty$.
\end{proof}

As in Lemma \ref{limit equ link}, let us consider the
family
\[
\Ga_\mu=\big\{
\nu\in C([0,1],E^+):\, \nu(0)=0,\ J_\mu(\nu(1))<0 \big\} \,,
\]
and the minimax schemes
\[
d_\mu^1=\inf_{\nu\in\Ga_\mu}\max_{t\in[0,1]}
J_\mu(\nu(t))  \quad
\text{and} \quad
d_\mu^2=\inf_{u\in E^+\setminus\{0\}}\max_{t\geq0}
J_\mu(tu) \,.
\]

\begin{Lem}\label{asym}
For the asymptotically linear equation \eqref{limit equ1},
there holds:
\begin{itemize}
\item[$(1)$] $\msk_\mu\setminus\{0\}\not=\emptyset$, $\ga_\mu>0$
 and $\msk_\mu\subset\cap_{q\geq2}W^{1,q}$;
\item[$(2)$] $\ga_\mu$ is attained and $\ga_\mu=d_\mu^1=d_\mu^2$;
\item[$(3)$] if $\mu_1>\mu_2$, then $\ga_{\mu_1}>\ga_{\mu_2}$.
\end{itemize}
\end{Lem}
\begin{proof}
Since $(1)$ is a direct consequence of \cite[Theorem 7.3]{Ding2}, we only
need to prove $(2)$ and $(3)$.

To show $(2)$, assume $\{u_n\}\subset\msk_\mu\setminus\{0\}$ such that
$\mst_\mu(u_n)\to\ga_\mu$. Clearly $\{u_n\}$ is a $(C)_c$ sequence, hence
is bounded. As is proved in \cite{Ding2}, $\{u_n\}$ is non-vanishing.
Since $\mst_\mu$ is $\Z^3$-invariant, up to a translation, we can assume
$u_n\rightharpoonup u\in\msk_\mu\setminus\{0\}$. Observe that, by Fatou's
lemma,
\[
\aligned
\ga_\mu&\leq
\mst_\mu(u)=\mst_\mu(u)-\frac12\mst_\mu'(u)u=\int \widehat G(|u|) \\
&\leq \liminf_{n\to\infty}\int \widehat G(|u_n|)=\liminf_{n\to\infty}
\Big(\mst_\mu(u_n)
-\frac12\mst_\mu'(u_n)u_n \Big)  \\
& =\ga_\mu \,,
\endaligned
\]
we find $\ga_\mu$ is attained. By noting that $\ga_\mu$ is also the least
energy of $J_\mu$, it is standard to check that
$\ga_\mu\leq d_\mu^1\leq d_\mu^2$.
To prove $d_\mu^2\leq\ga_\mu$ we first note
that if $s>0$,  by virtue of
$(g_1)$,
$g'(s)s>0$. Hence, if $u\in E\setminus\{0\}$
and $v\in E$, we have
\[
\aligned
&\big( \Psi''(u)[u,u]-\Psi'(u)u \big)
+2\big( \Psi''(u)[u,v]-\Psi'(u)v \big)
+\Psi''(u)[v,v] \\
=&\,\int g(|u|)|v|^2 + \int  g'
(|u|)|u| \bigg( |u| + \frac{\Re u\cdot v}{|u|} \bigg)^2
>0 \,.
\endaligned
\]
As a consequence of \cite[Theorem 5.1]{Ackermann}, if
$z\in E^+\setminus\{0\}$ satisfies $J_\mu'(z)z=0$ then
$J_\mu''(z)[z,z]<0$. Therefore, let $u\in E^+\setminus\{0\}$,
we find the function $t \mapsto J_\mu(tu)$
has at most one nontrivial critical point $t=t(u)>0$.
So, denoted by
\[
\msm_\mu:=\big\{
t(u)u: \, u\in E^+\setminus\{0\}, \ t(u)<\infty \big\} \,,
\]
we have $\msm_\mu\neq\emptyset$ due to $\ga_\mu$ is attained.
Meanwhile, we notice
\[
d_\mu^2=\inf_{z\in \msm_\mu}J_\mu(z) \,.
\]
Hence we have $d_\mu^2\leq\ga_\mu$ since $u^+\in\msm_\mu$ provided
$u\in\msr_\mu$.

Finally, $(3)$ comes directly because, if $u\in\msr_{\mu_1}$,
we already have  $u^+$
is a critical point of $J_{\mu_1}$ and
$\ga_{\mu_1}=J_{\mu_1}(u^+)=\max_{t\geq0}J_{\mu_1}(tu^+)$. Let $\tau>0$
such that $J_{\mu_2}(\tau u^+)=
\max_{t\geq0}J_{\mu_2}(t u^+)$, we deduce
\[
\aligned
\ga_{\mu_1}&=J_{\mu_1}(u^+)=\max_{t\geq0}J_{\mu_1}(tu^+)\\
&\geq J_{\mu_1}(\tau u^+) = \mst_{\mu_1}\big( \tau u^++
\msj_{\mu_1}(\tau u^+) \big) \\
&\geq \mst_{\mu_1}\big( \tau u^++ \msj_{\mu_2}(\tau u^+) \big) \\
&\geq \mst_{\mu_2}
\big( \tau u^++ \msj_{\mu_2}(\tau u^+) \big) +
\frac{\mu_1-\mu_2}2 \big| \tau u^++ \msj_{\mu_2}(\tau u^+)
\big|_2^2  \\
&= J_{\mu_2}(\tau u^+) + \frac{\mu_1-\mu_2}2
\big| \tau u^+ + \msj_{\mu_2}(\tau u^+)
\big|_2^2  \\
&\geq \ga_{\mu_2} + \frac{\mu_1-\mu_2}2
\big| \tau u^+ + \msj_{\mu_2}(\tau u^+)
\big|_2^2 \,,
\endaligned
\]
which ends the proof.
\end{proof}

\subsection{Some technical results}

We remark that, by $(V_1)$,
$V_\vr(x)\to V(0)$ uniformly on bounded sets of
$\R^3$ as $\vr\to0$. We will make use of this property and
the results just proved in Subsection \ref{TLE}
to prove some technical results that seem to be
useful in the sequel.

Denote $V_0=V(0)$, set
$V^0(x)=V(x)-V(0)$ and $V_\vr^0(x)=V^0(\vr x)$, we
soon have
\begin{\equ}\label{h-vr identity}
\widetilde\Phi_\vr(u)=\mst_{\mbox{\tiny $V_0$}}(u)+\frac12\int V_\vr^0(x)|u|^2
-\int \big( F(\vr x,|u|)-G(|u|) \big) \,.
\end{\equ}

\begin{Lem}\label{h-vr to h-0}
Let $(f_1)$-$(f_5)$
be satisfied, for each $u\in E^+$,
we have $h_\vr(u)\to \msj_{\mbox{\tiny $V_0$}}(u)$ as $\vr\to0$.
\end{Lem}
\begin{proof}
By \eqref{h-vr identity}, we deduce that
\begin{\equ}\label{A4}
\aligned
&\big( \widetilde\Phi_\vr(z_\vr)-\widetilde\Phi_\vr(w) \big)
+\big( \mst_{\mbox{\tiny $V_0$}}(w)-\mst_{\mbox{\tiny $V_0$}}(z_\vr) \big) \\
=&\,\frac12 \int V_\vr^0(x)\big( |z_\vr|^2-|w|^2 \big)
+ \int \big( G(|z_\vr|)-G(|w|) \big)\\
&\, -\int \big( F(\vr x,|z_\vr|)-F(\vr x,|w|) \big)
\endaligned
\end{\equ}
where $z_\vr=u+h_\vr(u)$, $w=u+\msj_{\mbox{\tiny $V_0$}}(u)$.
Denoted by $v_\vr=z_\vr-w$, we find
\[
\int V_\vr^0(x)\big( |z_\vr|^2-|w|^2 \big)
=\int V_\vr^0(x)|v_\vr|^2+2\Re\int V_\vr^0(x)w\cdot \ov{ v_\vr}
\]
and
\[
\aligned
&\int \big( G(|z_\vr|)-G(|w|) \big)
-\int \big( F(\vr x,|z_\vr|)-F(\vr x,|w|) \big)  \\
=&\, \Re\int g(|w|)w\cdot \ov{v_\vr} -
\Re\int f(\vr x,|w|)w\cdot \ov{v_\vr} \\
&\, +\int_0^1(1-s)\Psi''(w+s v_\vr)[v_\vr,v_\vr]\, ds \\
&\, -\int_0^1(1-s)\Psi_\vr''(w+s v_\vr)[v_\vr,v_\vr]\, ds  \,.
\endaligned
\]
Remark that, similar to \eqref{identity1} and
\eqref{identity2}, we infer
\[
\aligned
 &\int_0^1(1-s)\Psi_\vr''(z_\vr-s v_\vr)[v_\vr,v_\vr] \,ds \\
 &+\frac{1}{2}\|v_\vr\|^2+\frac12 \int V_\vr(x)|v_\vr|^2
 =\widetilde\Phi_\vr(z_\vr)-\widetilde\Phi_\vr(w)
\endaligned
\]
and
\[
\aligned
 &\int_0^1(1-s)\Psi''(w+sv_\vr)[v_\vr,v_\vr] \,ds \\
 &+\frac{1}{2}\|v_\vr\|^2+\frac{V_0}2 |v_\vr|_2^2
 =\mst_{\mbox{\tiny $V_0$}}(w)-\mst_{\mbox{\tiny $V_0$}}(z_\vr) \,.
\endaligned
\]
Then we get from \eqref{A4} (jointly with the definition of $f$
and \eqref{N2})
\begin{\equ}\label{A5}
\aligned
\|v_\vr\|^2+V_0|v_\vr|_2^2
\leq&\,\Re\int V_\vr^0(x)w\cdot \ov{v_\vr}
 + \Re\int  g(|w|)w\cdot \ov{v_\vr}  \\
 &\,- \Re\int f(\vr x,|w|)w\cdot \ov{ v_\vr } \\
\leq&\,\int |V_\vr^0(x)|\cdot |w|\cdot |v_\vr|
 + c_1\int_{\R^3\setminus\Lam_\vr} |w| \cdot |v_\vr|  \\
&\, + c_1\int_{\R^3\setminus\Lam_\vr} |w|^{p-1} \cdot |v_\vr| \\
&\, + \frac{a-|V|_\infty}2
\int_{\R^3\setminus\Lam_\vr} |w|\cdot |v_\vr|  \\
\leq&\, \Big(
\int |V_\vr^0(x)|^2 |w|^2 \Big)^{1/2} |v_\vr|_2 \\
&\,+ c_2\Big(
\int_{\R^3\setminus\Lam_\vr} |w|^2 \Big)^{1/2}|v_\vr|_2 \\
&\, + c_1\Big(
\int_{\R^3\setminus\Lam_\vr}|w|^p \Big)^{(p-1)/p} |v_\vr|_p \,.
\endaligned
\end{\equ}
Since $V_\vr^0(x)\to0$ uniformly on bounded sets of $\R^3$
as $\vr\to0$, we easily have
\[
\int |V_\vr^0(x)|^2 |w|^2 = o(1) \,.
\]
Moreover, by noting that
$w$ decays at infinity in the sense for $q=2,p$,
\[
\limsup_{R\to\infty}\int_{|x|\geq R}|w|^q=0 \,.
\]
We find (due to $0\in \Lam$)
\[
\int_{\R^3\setminus\Lam_\vr} |w|^2=o(1) \,,
\]
\[
\int_{\R^3\setminus\Lam_\vr}|w|^p=o(1) \,,
\]
as $\vr\to0$. Thus \eqref{A5} leads to
$\|v_\vr\|=\|h_\vr(u)-\msj_{\mbox{\tiny $V_0$}}(u)\|=o(1)$ as $\vr\to0$.
So we have the lemma proved.
\end{proof}

\begin{Lem}\label{MPG}
Assume that $(f_1)$-$(f_6)$ are satisfied,
for $\vr>0$ small enough, $I_\vr$ possesses the mountain-pass structure:
\begin{itemize}
\item[$(1)$] $I_\vr(0)=0$ and there exist $r>0$ and $\tau>0$
(both independent of $\vr$) such that
$I_\vr|_{S_r^+}\geq\tau$.
\item[$(2)$] there exists $u_0\in E^+$ (independent of $\vr$)
such that $\|u_0\|>r$ and $I_\vr(u_0)<0$.
\end{itemize}
\end{Lem}
\begin{proof}
Since we have $I_\vr(u)\geq\widetilde\Phi_\vr(u)$ for all $u\in E^+$,
$(1)$ follows easily from Lemma \ref{link1}.

To check $(2)$, let $w=w^++w^-\in \msr_{\mbox{\tiny $V_0$}}$ be the least energy
solution to
\[
-i\al\cdot\nabla u + a\bt u + V_0u = g(|u|)u
\]
with $|w(0)|=\max_{\R^3}|w(x)|$.
Following Lemma \ref{limit equ link} $(2)$ and Lemma \ref{asym}
$(2)$, we have
\[
\ga_{\mbox{\tiny $V_0$}}=\inf_{\ga\in\Ga_0}\max_{t\in[0,1]}J_{\mbox{\tiny $V_0$}}(\ga(t))
=\inf_{e\in E^+\setminus\{0\}}\max_{t\geq0}J_{\mbox{\tiny $V_0$}}(tu)  \,,
\]
where $\Ga_0:=\big\{\ga\in C([0,1],E^+):\,
\ga(0)=0,\ J_{\mbox{\tiny $V_0$}}(\ga(1))<0 \big\}$.
From Lemma \ref{limit equ link} $(1)$ and
Lemma \ref{ding book lemma 7.7} $(2)$,
we see that there exists $t_0>0$
(large enough) such that
\[
\aligned
J_{\mbox{\tiny $V_0$}}(t_0 w^+)
=&\,\frac12\big( \|t_0 w^+\|^2-\|\msj_{\mbox{\tiny $V_0$}}(t_0 w^+)\|^2
\big) + \frac{V_0}2\int|t_0 w^++\msj_{\mbox{\tiny $V_0$}}(t_0 w^+)|^2 \\
&\,-\int G \big( |t_0 w^++\msj_{\mbox{\tiny $V_0$}}(t_0 w^+)| \big)
<-1 \,.
\endaligned
\]
Hence, there is $R_0>0$ such that
\begin{\equ}\label{X3}
\aligned
&\frac12\big( \|t_0 w^+\|^2-\|\msj_{\mbox{\tiny $V_0$}}(t_0 w^+)\|^2
\big) + \frac{V_0}2\int|t_0 w^++\msj_{\mbox{\tiny $V_0$}}(t_0 w^+)|^2 \\
&-\int_{B_{R_0}} G \big( |t_0 w^++\msj_{\mbox{\tiny $V_0$}}(t_0 w^+)| \big)
\leq -\frac12 \,.
\endaligned
\end{\equ}
Recall that $V_\vr(x)\to V_0$ uniformly on bounded
sets of $\R^3$, it follows from Lemma \ref{h-vr to h-0}
and \eqref{X3} that
\[
\aligned
I_\vr(t_0 w^+)=&\,\frac12\big( \|t_0 w^+\|^2-\|h_\vr(t_0 w^+)\|^2
\big) + \frac12\int V_\vr(x)|t_0 w^++h_\vr(t_0 w^+)|^2 \\
&-\int F\big( \vr x, |t_0 w^++h_\vr(t_0 w^+)| \big)  \\
\leq&\,\frac12\big( \|t_0 w^+\|^2-\|h_\vr(t_0 w^+)\|^2
\big) + \frac12\int V_\vr(x)|t_0 w^++h_\vr(t_0 w^+)|^2 \\
&-\int_{\Lam_\vr} G\big( |t_0 w^++h_\vr(t_0 w^+)| \big) \\
\leq&\,\frac12\big( \|t_0 w^+\|^2-\|\msj_{\mbox{\tiny $V_0$}}(t_0 w^+)\|^2
\big) + \frac{V_0}2\int |t_0 w^++\msj_{\mbox{\tiny $V_0$}}(t_0 w^+)|^2 \\
&-\int_{B_{R_0}} G \big( |t_0 w^++\msj_{\mbox{\tiny $V_0$}}(t_0 w^+)|
\big) +o(1)   \\
\leq&\, -\frac12 + o(1) \qquad \text{as } \vr\to0 \,.
\endaligned
\]
Therefore, there is $\vr_0>0$ such that
$I_\vr(t_0 w^+)<0$ for all $\vr\in(0,\vr_0]$, ends the proof.
\end{proof}

\begin{Lem}\label{I-vr Cc}
Assuming $(f_1)$-$(f_6)$, for each $\vr>0$,
$I_\vr$ satisfies the $(C)_c$-condition.
\end{Lem}
\begin{proof}
Firstly, it follows from the definition of $h_\vr$ that
\begin{\equ}\label{identity3}
\widetilde\Phi_\vr'\big( u+h_\vr(u) \big)z=0
\quad \text{for all } u\in E^+ \text{ and } z\in E^- \,.
\end{\equ}
Hence a direct calculation shows
\begin{\equ}\label{identity4}
\aligned
I_\vr'(u)u&=\widetilde\Phi_\vr'\big( u+h_\vr(u) \big)
(u+h_\vr'(u)u)  \\
&=\widetilde\Phi_\vr'\big( u+h_\vr(u) \big)
(u+h_\vr(u))  \\
&=\widetilde\Phi_\vr'\big( u+h_\vr(u) \big)
(u-h_\vr(u))  \,.
\endaligned
\end{\equ}

Now let $\{w_n\}\subset E^+$ be a $(C)_c$ sequence for $I_\vr$
and set $u_n:=w_n+h_\vr(w_n)$, we check easily (from the proof
of Lemma \ref{C-c condition}) that $\{u_n\}$ possesses a convergent
subsequence. Therefore, we have $I_\vr$ satisfies the $(C)_c$ condition.
\end{proof}

Define
\[
c_\vr:=\inf_{\nu\in\Ga_\vr}\max_{t\in[0,1]}I_\vr(\nu(t)) \,,
\]
where $\Ga_\vr:=\big\{\nu\in C([0,1],E^+):\, \nu(0)=0, \
I_\vr(\nu(1))<0 \big\}$.
Then $\tau\leq c_\vr<\infty$
is a well-defined critical value for $I_\vr$
(also for $\widetilde\Phi_\vr$).

\begin{Lem}\label{c-vr=d-vr}
$c_\vr=\inf_{u\in E^+\setminus\{0\}}\max_{t\geq0}I_\vr(tu)$.
\end{Lem}
\begin{proof}
Indeed,
set $d_\vr=\inf_{u\in E^+\setminus\{0\}}\max_{t\geq0}I_\vr(tu)$,
we have $d_\vr\geq c_\vr$
by virtue of $(f_6)$ and the proof of Lemma
\ref{ding book lemma 7.7} $(2)$. To prove the other inequality we first note
that if $s>0$,  by virtue of
$(g_1)$, $(g_3)$ and the definition of $f$,
$f_s(x,s)s>0$. Hence, if $u\in E\setminus\{0\}$
and $v\in E$, we have
\[
\aligned
&\big( \Psi_\vr''(u)[u,u]-\Psi_\vr'(u)u \big)
+2\big( \Psi_\vr''(u)[u,v]-\Psi_\vr'(u)v \big)
+\Psi_\vr''(u)[v,v] \\
=&\,\int f(\vr x, |u|)|v|^2 + \int  f_s
(\vr x, |u|)|u| \bigg( |u| + \frac{\Re u\cdot v}{|u|} \bigg)^2
>0 \,.
\endaligned
\]
As proved in \cite{Ackermann}, if
$z\in E^+\setminus\{0\}$ satisfies $I_\vr'(z)z=0$ then
$I_\vr''(z)[z,z]<0$. Therefore, let $u\in E^+\setminus\{0\}$,
we find the function $t \mapsto I_\vr(tu)$
has at most one nontrivial critical point $t=t(u)>0$.
So, denoted by
\[
\msn:=\big\{ t(u)u: \, u\in E^+\setminus\{0\}, \ t(u)<\infty \big\} \,,
\]
we have $\msn\neq\emptyset$ due to Lemma \ref{MPG}
and Lemma \ref{I-vr Cc} (In general, we remark that $\msn$ is not the Nehari manifold
since $\msn$ is not defined for all directions in $E^+$ for the
asymptotically linear case). Meanwhile, we notice
\[
d_\vr=\inf_{z\in \msn}I_\vr(z) \,.
\]
Thus we only need to show that given $\nu\in\Ga_\vr$ there exists
$\bar t\in[0,1]$ such that $\nu(\bar t)\in\msn$. Assuming contrarily we
have $\nu([0,1])\cap\msn=\emptyset$. In virtue of $(f_1)$ and
Lemma \ref{link1}
\[
I_\vr'(\nu(t))\nu(t)>0 \quad \text{for } t>0 \text{ small} \,.
\]
Since the function $t\mapsto I_\vr'(\nu(t))\nu(t)$ is continuous and
$I_\vr'(\nu(t))\nu(t)\neq0$ for all $t\in(0,1]$, we have
\[
I_\vr'(\nu(t))\nu(t)>0 \quad \text{for all } t\in(0,1] \,.
\]
Then, we find for all $t\in[0,1]$
\[
\aligned
I_\vr(\nu(t))&=\frac12 I_\vr'(\nu(t))\nu(t) + \int \widehat
F\big( \vr x, |\nu(t)+h_\vr(\nu(t))| \big)  \\
&\geq \frac12 I_\vr'(\nu(t))\nu(t) > 0 \,,
\endaligned
\]
and this contradicts the definition of $\Ga_\vr$.
Consequently, by noting that
$\nu(t)$ crosses $\msn$ provided $\nu\in\Ga_\vr$,
we have $d_\vr\leq c_\vr$.
\end{proof}

\begin{Lem}\label{c-vr leq}
$c_\vr\leq \ga_{\mbox{\tiny $V_0$}} +o(1)$ as $\vr\to0$.
\end{Lem}
\begin{proof}
Again, let $w=w^++w^-\in\msr_{\mbox{\tiny $V_0$}}$, set $t_0>0$ such that
$J_{\mbox{\tiny $V_0$}}(t_0 w^+)\leq-1$. By virtue of
Lemma \ref{h-vr to h-0} and Lemma \ref{c-vr=d-vr}, it sufficient to prove
\begin{\equ}\label{Ascoli}
I_\vr(tw^+)=J_{\mbox{\tiny $V_0$}}(t w^+) + o(1)
\quad \text{uniformly in } t\in[0,t_0]
\end{\equ}
as $\vr\to0$.

To this end, we only need to show the family
$\{H_\vr\}\subset C([0,t_0])$
\begin{\equ}\label{H-vr}
H_\vr:[0,t_0]\to\R, \quad t\mapsto
I_\vr(tw^+)-J_{\mbox{\tiny $V_0$}}(t w^+)
\end{\equ}
is equicontinuous.
Observe that the boundedness of $h_\vr$ and $\widetilde\Phi_\vr''$
imply the boundedness of $h_\vr'$ due to \eqref{N6} and \eqref{N7},
we conclude the derivatives of the family defined in \eqref{H-vr}
are uniformly bounded. Then the proof ends with a trivial application of Arzel\`{a}-Ascoli theorem.
\end{proof}

\section{Proof of the main results}\label{PMT}

We are now presenting the proof of the main results on
the nonlinear Dirac equation:
\begin{\equ}\label{D3}
-i\al\cdot\nabla u + a\bt u + V_\vr(x) u = g(|u|)u \,.
\end{\equ}
Except for the hypotheses in Theorem \ref{main theorem} and
Theorem \ref{main theorem 2},
without loss of generality, we may assume that the boundary of
$\Lam$ is smooth, and that $0\in\Lam$ such that
$V(0)=V_0:=\min_\Lam V$.

Thanks to the preparatory results already proved in Section
\ref{auxiliary results}, we will give an unified prove cover
both super-linear and asymptotically linear cases.
As mentioned before, in order to localize the desired
solutions, we consider the modification of the function $g$
given by $f$ in \eqref{mod} and the associated equation
\begin{\equ}\label{PT1}
-i\al\cdot\nabla u + a\bt u + V_\vr(x) u = f(\vr x, |u|)u \,.
\end{\equ}
For ease of notations, let us denote
\[
\msa=\big\{x\in\Lam:\, V(x)=V_0 \big\} \,.
\]
And for the later use, letting $D=-i\al\cdot\nabla$, we rewrite
\eqref{PT1} as
\[
Du=-a\bt u - V_\vr(x)u + f(\vr x,|u|)u \,.
\]
Acting the operator $D$ on the two sides of the above
representation and noting that $D^2=-\De$, we find
\[
\De u = \big( a^2-V_\vr^2(x) \big) u - f^2(\vr x,|u|) u +
D\big( V_\vr(x)-f(\vr x,|u|) \big)u \,.
\]
Letting
\[
\sgn\, u =\left\{
\aligned
&\frac{\bar u}{|u|} \quad {\rm if\ } u\neq0, \\
&0 \quad {\rm if\ } u=0,
\endaligned \right.
\]
by the Kato's inequality \cite{Dautray}, there holds
\[
\De |u|\geq \Re[\De u \cdot (\sgn\, u)] \,.
\]
Observe that
\[
\Re \Big[ D\big( V_\vr(x)-f(\vr x,|u|) \big)u \cdot \frac{\bar u}{|u|}
\Big]=0 \,,
\]
hence
\begin{\equ}\label{PT4}
\De |u|\geq\big( a^2-V_\vr^2(x) \big) |u| - f^2(\vr x,|u|) |u| \,.
\end{\equ}
We remind that \eqref{PT4} together with the regularity results for $u$ (see Lemma \ref{critical points in W1,q})
imply there is $M>0$ (independent of $\vr$) satisfying
\[
\Delta |u| \geq -M |u|.
\]
It then follows from the sub-solution estimate \cite{Trudinger, Simon} that
\begin{\equ}\label{sub solu}
|u(x)|\leq  C_0 \int_{B_1(x)} |u(y)| dy
\end{\equ}
with $C_0>0$ independent of $x$, $\vr$ and $u\in\msl_\vr$.

\begin{Lem}\label{concentration}
Assuming $(f_1)$-$(f_6)$ and ,
for all $\vr$ sufficiently small, let $u_\vr\in\msl_\vr$,
then $|u_\vr|$ possesses a
(global) maximum
$x_\vr\in\Lam_\vr$ such that
\[
\lim_{\vr\to0}V(\vr x_\vr)= V_0 = \min_{x\in\Lam}V(x) \,.
\]
Moreover, by setting $v_\vr(x)=u_\vr(x+x_\vr)$, we must have
$|v_\vr|$ decays uniformly at infinity and
$\{v_\vr\}$ converges in $H^1$ to a ground state solution to
\[
-i\al\cdot\nabla v + a\bt v + V_0 v = g(|v|)v \,.
\]
\end{Lem}

\begin{proof}
Let $u_\vr\in E$ be the critical point so that
$\widetilde\Phi_\vr(u_\vr)=c_\vr$. We have $\{u_\vr\}$
is a bounded set in $E$.

{\it Step 1. } $\{u_\vr\}$ is non-vanishing.

Suppose contrarily that
\[
\sup_{x\in\R^3}\int_{B_R(x)}|u_\vr|^2 \,dx \to0
\quad \text{as } \vr\to0
\]
for all $R>0$. Then, by Lion's concentration principle \cite{Lions},
$|u_\vr|_q\to0$ for $q\in(2,3)$. Since $\{u_\vr\}$ is bounded in
$E$, we have ${\it meas}\{x\in\R:\, |u_\vr(x)|\geq r\}$ is uniformly
bounded for all
$\vr>0$ provided $r>0$ is fixed. So we find
\[
\int_{\{x\in\R:\, |u_\vr(x)|\geq r\}}|u_\vr|^2 \,dx\to0
\]
as $\vr\to0$. Now, we have
\[
c_\vr=\widetilde\Phi_\vr(u_\vr)-\frac12\widetilde\Phi_\vr'(u_\vr)u_\vr
=\int \widehat F(\vr x, |u_\vr|)=o(1)  \,,
\]
which contradict to the fact $c_\vr\geq\tau>0$ (see Lemma \ref{MPG}).

{\it Step 2.} $\{\chi_{\mbox{\tiny $\Lam_\vr$}}
\cdot u_\vr\}$ is non-vanishing.

Indeed, if $\{\chi_{\mbox{\tiny $\Lam_\vr$}}
\cdot u_\vr\}$ vanishes, by virtue of {\it Step 1} we have
$\{(1-\chi_{\mbox{\tiny $\Lam_\vr$}})
\cdot u_\vr\}$ is non-vanishing, that is (together with the fact
that each $u_\vr$ is decaying at infinity) there exist $x_\vr\in\R^3$
and constants $R>0$ and $\de>0$
such that $B_R(x_\vr)\subset\R^3\setminus\Lam_\vr$ and
\[
\int_{B_R(x_\vr)}|u_\vr|^2 \geq \de \,.
\]
Set $v_\vr(x)=u_\vr(x+x_\vr)$, then $v_\vr$ satifies
\begin{\equ}\label{v-vr solves}
-i\al\cdot\nabla v_\vr + a\bt v_\vr + \hat V_\vr(x) v_\vr
=f(\vr(x+x_\vr), |v_\vr|)v_\vr \,,
\end{\equ}
where $\hat V_\vr(x):=V(\vr(x+x_\vr))$. Additionally,
$v_\vr\rightharpoonup v$ in $E$ and $v_\vr\to v$ in
$L_{loc}^q$ for $q\in[1,3)$. Now assume without loss of
generality that $V(\vr x_\vr)\to V_\infty$, using
$\psi\in C_c^\infty(\R^3,\C^4)$ as a test function in
\eqref{v-vr solves}, one gets
\[
\aligned
0&=\lim_{\vr\to0}
\int \Big(-i\al\cdot\nabla v_\vr + a\bt v_\vr + \hat V_\vr(x) v_\vr
-f(\vr(x+x_\vr), |v_\vr|)v_\vr \Big) \ov\psi  \\
&=\int \Big(-i\al\cdot\nabla v + a\bt v + V_\infty v
-\tilde g( |v|)v \Big)\ov\psi \,.
\endaligned
\]
Hence $v$ satisfies
\begin{\equ}\label{v solves}
-i\al\cdot\nabla v + a\bt v + V_\infty v
=\tilde g( |v|)v \,.
\end{\equ}
However, using the test function $v^+-v^-$ in
\eqref{v solves}, we have (with $(f_3)$)
\[
\aligned
0&=\|v\|^2+V_\infty\int v\cdot \ov{(v^+-v^-)} -
\int \tilde g(|v|)v\cdot\ov{(v^+-v^-)} \\
&\geq \|v\|^2 - \frac{|V|_\infty}a \|v\|^2 -
\frac{a-|V|_\infty}{2a}\|v\|^2 \\
&=\frac{a-|V|_\infty}{2a}\|v\|^2 \,.
\endaligned
\]
Therefore, we have $v=0$ a contradiction.

{\it Step 3.} Let $x_\vr\in \R^3$ and $R,\de>0$ be such that
\[
\int_{B_R(x_\vr)}|\chi_{\mbox{\tiny $\Lam_\vr$}}\cdot u_\vr|^2
\geq \de \,.
\]
Then $\vr x_\vr\to \msa$.

To prove this, clearly, we may first choose
$x_\vr\in\Lam_\vr$, i.e. $\vr x_\vr\in\Lam$. Suppose that,
up to a subsequence if necessary,
$\vr x_\vr\to x_0\in\bar\Lam$ as $\vr\to0$. Again, set
$v_\vr(x)=u_\vr(x+x_\vr)$, we have
$v_\vr\rightharpoonup v$ in $E$ and $v$ satisfies
\begin{\equ}\label{limit equ2}
-i\al\cdot\nabla v + a\bt v + V(x_0) v = f_\infty(x, |v|)v \,,
\end{\equ}
where $f_\infty(x,s)=\chi_\infty\cdot g(s)+
(1-\chi_\infty)\cdot\tilde g(s)$
and $\chi_\infty$ is either a characteristic function of a half-space
of $\R^3$ provided
\[
\limsup_{\vr\to0}\dist(x_\vr,\pa\Lam_\vr)<+\infty
\]
or $\chi_\infty\equiv1$ (since $\Lam$ is an open set with smooth boundary,
this can be see by the fact $\chi_\Lam(\vr(\cdot + x_\vr ))$
converges pointwise a.e. on $\R^3$ to $\chi_\infty(\cdot)$ and $x_\vr\in\Lam_\vr$).
Denote $S_\infty$ to be the associate energy functional to \eqref{limit equ2}:
\[
S_\infty(u):=\frac12\big( \|u^+\|^2-\|u^-\|^2 \big) + \frac {V(x_0)}2|u|_2^2
- \Psi_\infty(u) \,,
\]
where
\[
\Psi_\infty(u):=\int F_\infty(x,|u|) \quad \text{and } \quad
F_\infty(x,s)=\int_0^s f_\infty(x,\tau)\tau\,d\tau \,.
\]
By noting that $\Psi_\infty(u)\leq \Psi(u)$ (thanks to $(f_2)$),
we have
\[
S_\infty(u)\geq \mst_{\mbox{\tiny $V(x_0)$}}(u)
= \mst_{\mbox{\tiny $V_0$}}(u) + \frac{V(x_0)-V_0}2|u|_2^2
\quad \text{for all } u\in E \,.
\]
Furthermore, if $s>0$ we find from
$(g_1)$, $(g_3)$ and the definition of $\tilde g$
that $\tilde g'(s)s>0$. Hence, if $u\in E\setminus\{0\}$
and $v\in E$, we have
\[
\aligned
&\big( \Psi_\infty''(u)[u,u]-\Psi_\infty'(u)u \big)
+2\big( \Psi_\infty''(u)[u,v]-\Psi_\infty'(u)v \big)
+\Psi_\infty''(u)[v,v] \\
=&\,\int f_\infty(x, |u|)|v|^2 + \int \pa_s f_\infty
( x, |u|)|u| \bigg( |u| + \frac{\Re u\cdot v}{|u|} \bigg)^2
>0 \,.
\endaligned
\]
Now let us define (as before)
$h_\infty:E^+\to E^-$ and $I_\infty:E^+\to\R$ by
\[
S_\infty\big( u+h_\infty(u) \big)=\max_{v\in E^-}S_\infty(u+v) \,,
\]
\[
I_\infty(u)=S_\infty\big( u+h_\infty(u) \big) \,.
\]
It is standard to see that: if $z\in E^+\setminus\{0\}$
satisfies $I_\infty'(z)z=0$, then $I_\infty''(z)[z,z]<0$
(see \cite{Ackermann}). Since we already have $v\neq0$
is a critical point of $S_\infty$, we then infer $v^+$
is a critical point of $I_\infty$ and
$I_\infty(v^+)=\max_{t\geq0}I_\infty(tv^+)$. Let $\tau>0$
such that $J_{\mbox{\tiny $V_0$}}(\tau v^+)=
\max_{t\geq0}J_{\mbox{\tiny $V_0$}}(t v^+)$, we infer
\begin{\equ}\label{X1}
\aligned
S_\infty(v)&=I_\infty(v^+)=\max_{t\geq0}I_\infty(tv^+)\\
&\geq I_\infty(\tau v^+) = S_\infty\big( \tau v^+
+ h_\infty(\tau v^+) \big) \\
&\geq S_\infty\big( \tau v^+ + \msj_{\mbox{\tiny $V_0$}}(\tau v^+) \big) \\
&\geq \mst_{\mbox{\tiny $V_0$}}
\big( \tau v^+ + \msj_{\mbox{\tiny $V_0$}}(\tau v^+) \big) +
\frac{V(x_0)-V_0}2 \big| \tau v^+ + \msj_{\mbox{\tiny $V_0$}}(\tau v^+)
\big|_2^2  \\
&\geq \ga_{\mbox{\tiny $V_0$}} + \frac{V(x_0)-V_0}2 \big| \tau v^++ \msj_{\mbox{\tiny $V_0$}}(\tau v^+)
\big|_2^2 \,.
\endaligned
\end{\equ}
On the other hand, by Fatou's lemma, we deduce
\[
\aligned
c_\vr&=\widetilde\Phi_\vr(u_\vr)-\frac12\widetilde\Phi_\vr'(u_\vr)u_\vr \\
&=\int \widehat F(\vr x, |u_\vr|)
= \int \widehat F(\vr (x+x_\vr), |v_\vr|) \\
&\geq \int \widehat F_\infty(x,|v|)  \\
&=S_\infty(v)-\frac12 S_\infty'(v)v = S_\infty(v) \,,
\endaligned
\]
where $\widehat F_\infty(x,s):=\frac12f_\infty(x,s)s^2-F_\infty(x,s)$
for $(x,s)\in\R^3\times\R^+$. Therefore,
together with \eqref{X1}, we have $c_\vr\geq\ga_{\mbox{\tiny $V_0$}}$ and
$c_\vr>\ga_{\mbox{\tiny $V_0$}}$ provided $V(x_0)\neq V_0$. Therefore,
by virtue of Lemma \ref{c-vr leq}, we soon have $x_0\in\msa$ and
$\chi_\infty\equiv1$.

{\it Step 4.} Let $v_\vr$ be defined in {\it Step 3}, then
$v_\vr\to v$ in $E$.

It sufficient to prove that there is a subsequence $\{v_{\vr_j}\}$
such that $v_{\vr_j}\to v$ in $E$.
Recall that, as the argument shows, $v$ is a ground state
solution to
\begin{\equ}\label{X2}
-i\al\cdot\nabla v + a\bt v + V_0 v = g(|v|)v \,,
\end{\equ}
and
\[
\lim_{\vr\to0}\int \widehat F(\vr(x+x_{\vr}), |v_{\vr}|)
= \int  \widehat G (|v|) \,.
\]

Let $\eta:[0,\infty)\to[0,1]$ be a smooth function satisfying
$\eta(s)=1$ if $s\leq 1$, $\eta(s)=0$ if $s\geq2$. Define
$\tilde v_j(x)=\eta(2|x|/j)v(x)$. One has
\begin{\equ}\label{converge1}
\|\tilde v_j-v\|\to0 \quad
\text{and} \quad |\tilde v_j-v|_q\to0 \quad
\text{as } j\to\infty
\end{\equ}
for $q\in[2,3]$. Set $B_d:=\big\{x\in\R^3:\,|x|\leq d \big\}$ for
$d>0$. We have that there possesses a subsequence $\{v_{\vr_j}\}$
such that, for any $\de>0$ there exists $r_\de>0$ satisfying
\[
\limsup_{j\to\infty}\int_{B_j\setminus B_r}|v_{\vr_j}|^q\leq \de
\]
for all $r\geq r_\de$ (see an argument of \cite[Lemma 5.7]{Ding2}).
Here we will use
\[
q=\left\{
\aligned
&p \quad \text{for the super-linear case}, \\
&2 \quad \text{for the asymptotically linear case},
\endaligned \right.
\]
where $p\in(2,3)$ is the constant in condition $(g_2)(i)$.
Denote $z_j=v_{\vr_j}-\tilde v_j$,
we remark that $\{z_j\}$ is bounded in $E$ and
\begin{\equ}\label{converge2}
\aligned
\lim_{j\to\infty}&\bigg|
\int F\big(\vr_j(x+x_{\vr_j}),|v_{\vr_j}|\big)
-F\big(\vr_j(x+x_{\vr_j}),|z_j|\big)\\
&-F\big(\vr_j(x+x_{\vr_j}),|\tilde v_j|\big) \bigg| =0
\endaligned
\end{\equ}
and
\begin{\equ}\label{converge3}
\aligned
\lim_{j\to\infty}&\bigg|
\int \Big[
f\big(\vr_j(x+x_{\vr_j}),|v_{\vr_j}|\big)v_{\vr_j}
-f\big(\vr_j(x+x_{\vr_j}),|z_j|\big)z_j\\
&-f\big(\vr_j(x+x_{\vr_j}),|\tilde v_j|\big)\tilde v_j
\Big]\ov\va \bigg| =0
\endaligned
\end{\equ}
uniformly in $\va\in E$ with $\|\va\|\leq1$
(see \cite[Lemma 7.10]{Ding2}).
Using the decay of $v$ and the fact that $\hat V_{\vr_j}(x)\to V_0$,
$F(\vr_j(x+x_{\vr_j}),s)\to G(s)$ as $j\to\infty$
uniformly on any
bounded set of $x$, one checks easily the following
\[
\Re\int \hat V_{\vr_j}(x)v_{\vr_j}\cdot\tilde v_j \to
\int V_0 \cdot|v|^2 \,, \quad
\int F\big(\vr_j(x+x_{\vr_j}),|\tilde v_j|\big) \to
\int G(|v|) \,.
\]
Let us denote $\hat\Phi_\vr$ to be
the associate energy functional of \eqref{v-vr solves},
then we have
\[
\aligned
\hat\Phi_{\vr_j}(z_j)=&\,\hat\Phi_{\vr_j}(v_{\vr_j})-S_\infty(v) \\
&\, + \int F\big(\vr_j(x+x_{\vr_j}),|v_{\vr_j}|\big)
-F\big(\vr_j(x+x_{\vr_j}),|z_j|\big)  \\
&\,-F\big(\vr_j(x+x_{\vr_j}),|\tilde v_j|\big) +o(1) \\
=&\, o(1)
\endaligned
\]
as $j\to\infty$, which implies that $\hat\Phi_{\vr_j}(z_j)\to0$.
Similarly,
\[
\aligned
\hat\Phi_{\vr_j}'(z_j)\va= &\,
\Re\int \Big[
f\big(\vr_j(x+x_{\vr_j}),|v_{\vr_j}|\big)v_{\vr_j}
-f\big(\vr_j(x+x_{\vr_j}),|z_j|\big)z_j\\
&-f\big(\vr_j(x+x_{\vr_j}),|\tilde v_j|\big)\tilde v_j
\Big]\ov\va + o(1) \\
=&\, o(1)
\endaligned
\]
as $j\to\infty$ uniformly in $\|\va\|\leq1$, which implies
$\hat\Phi_{\vr_j}'(z_j)\to0$. Therefore,
\begin{\equ}\label{o11}
o(1)=\hat\Phi_{\vr_j}(z_j)-\frac12\hat\Phi_{\vr_j}'(z_j)z_j
=\int \widehat F\big( \vr_j(x+x_{\vr_j}), |z_j|) \,.
\end{\equ}
Owning to $(f_6)$ and the regularity result,
for any fixed $r>0$, one has
\[
\int \widehat F\big( \vr_j(x+x_{\vr_j}), |z_j|\big)
\geq C_r \int_{\{x\in\R^3:|z_j(x)|\geq r\}} |z_j|^2
\]
for some constant $C_r$ depends only on $r$. Hence
\[
\int_{\{x\in\R^3:|z_j(x)|\geq r\}} |z_j|^2\to0
\]
as $j\to\infty$ for any fixed $r>0$. Notice $\{|z_j|_\infty\}$
is bounded, as a consequence, we get
\[
\aligned
\Big( 1-\frac{|V|_\infty}a \Big) \|z_j\|^2 \leq&\,
\|z_j\|^2 + \Re\int\hat V_{\vr_j}(x)z_j\cdot\ov{(z_j^+-z_j^-)} \\
=&\,\hat\Phi_{\vr_j}'(z_j)(z_j^+-z_j^-) \\
&\,+ \Re\int
f\big(\vr_j(x+x_{\vr_j}),|z_j| \big)z_j\cdot\ov{(z_j^+-z_j^-)}\\
\leq&\, o(1)+ \frac{a-|V|_\infty}{2a}\|z_j\|^2 \\
&\, + C_\infty\int_{\{x\in\R^3:|z_j(x)|\geq r\}}
|z_j|\cdot |z_j^+-z_j^-|  \\
\leq&\, o(1)+\frac{a-|V|_\infty}{2a}\|z_j\|^2 \,,
\endaligned
\]
that is, $\|z_j\|\to0$ as $j\to\infty$. Together with \eqref{converge1}
we get $v_{\vr_j}\to v$ in $E$.

{\it Step 5.} $v_\vr(x)\to0$ as $|x|\to\infty$ uniformly for all small
$\vr$.

Assume by contradiction that there exist $\de>0$ and
$y_\vr \in\R^3$ with $|y_\vr|\to\infty$ such that
$$
\de \leq |v_\vr(y_\vr)| \leq C_0 \int_{B_1(y_\vr)} |v_\vr(y)| \,dy \, .
$$
Since $v_\vr\to v$ in $E$, we
obtain, as $\vr\to0$,
\[
\begin{aligned}
\de &\leq C_0\Big(\int_{B_1(y_\vr)}|v_\vr|^2\Big)^{1/2}\\
 &\leq C_0\Big(\int |v_\vr-v|^2\Big)^{1/2}+C_0\Big(\int_{B_1(y_\vr)}|v|^2\Big)^{1/2}\to0,
\end{aligned}
\]
a contradiction.

By virtue of {\it Step 5} it
is clear that one may assume the sequence $\{x_\vr\}$ in {\it Step 3}
to be the maximum
points of $|u_\vr|$. Moreover, from the above argument we readily
see that, any sequence of such points satisfies
$\vr x_\vr$ converging to some point in $\msa$ as
$\vr\to0$.

Finally, in order to verify that $v_\vr\to v$ in $H^1$,
we first deduce from
\eqref{v-vr solves} and \eqref{X2} that
\[
H_0\, (v_\vr-v)=f\big( \vr(x+x_\vr), |v_j| \big)v_j - g(|v|)v
-\big( \hat V_\vr(x)v_j-V(x_0)v \big) \,.
\]
Using {\it Step 4} and the uniform estimate in Remark
\ref{uniformly estimate}, it is easy to check that
$|H_0\,(v_\vr-v)|_2\to0$ as $\vr\to0$. Therefore
$v_j\to v$ in $H^1(\R^3,\C^4)$, ending the proof.
\end{proof}

The {\it Step 5} in the previous lemma shows an uniform
decay estimate, not surprisingly,
the decay rate can be shown to be exponential:

\begin{Lem}\label{exp decay}
There exist $C>0$ such that for all $\vr >0$ small
\[
|u_\vr (x)|\leq Ce^{-{c_0}|x-x_\vr|}
\]
where $c_0=\sqrt{\big(a^2-|V|_\infty^2\big)}$.
\end{Lem}
\begin{proof}
The uniform decay estimate together with \eqref{PT4}
allow us to take $R>0$ sufficiently large such that
\[
\Delta|v_\vr|\geq\big(a^2-|V|_\infty^2\big)|v_\vr|
\]
for all $|x|\geq R$ and $\vr >0$ small. Let
$\Ga(y)=\Ga(y,0)$ be a fundamental solution to
$-\Delta+\big(a^2-|V|_\infty^2\big)$. Using the uniform boundedness, we may
choose that $|v_\vr (y)|\leq\big(a^2-|V|_\infty^2\big)\Ga(y)$ holds
on $|y|=R$ for all $\vr >0$ small. Let $z_\vr =|v_\vr|-\big(
a^2-|V|_\infty^2\big)\Ga$. Then
\[
\begin{aligned}
\De z_\vr &=\Delta|v_\vr|-\big(a^2-|V|_\infty^2\big)\De\Ga\\
 &\geq\big(a^2-|V|_\infty^2\big)\big(|v_\vr|-(a^2-|V|_\infty^2)\Ga\big)
 =(a^2-|V|_\infty^2)z_\vr .
\end{aligned}
\]
By the maximum principle we can conclude that $z_\vr (y)\leq0$ on
$|y|\geq R$. It is well known that there is $C'>0$ such that
$\Ga(y)\leq C'\exp(-c_0|y|)$ on $|y|\geq1$, we see that
\[
|v_\vr (y)|\leq C''e^{-c_0|y|}
\]
for all $y\in\R^3$ and all $\vr >0$ small, that is
\[
|u_\vr (x)|\leq Ce^{-c_0 |x-x_\vr|}
\]
as claimed.
\end{proof}

Now, we are ready to prove our main theorems.

\begin{proof}[Unified proof of Theorem \ref{main theorem}
and Theorem \ref{main theorem 2}]
Define
\[
w_\vr(x)=u_\vr(x/\vr) \quad
\text{and} \quad y_\vr=\vr x_\vr \,.
\]
Then $w_\vr$ is a solution of
\[
-i\vr\al\cdot\nabla w + a\bt w + V(x) w = f(x, |w|)w
\]
for all $\vr>0$.
Since $y_\vr$ is a maximum point of $|w_\vr|$,
due to Lemma \ref{concentration} and Lemma \ref{exp decay},
we have
\[
|w_\vr (x)|\leq Ce^{-\frac{c_0}\vr |x-y_\vr|}
\]
and $y_\vr\to\msa$ as $\vr\to0$. From the assumption
\[
\min_\Lam V < \min_{\pa\Lam} V \,,
\]
we conclude that $\de:=\dist(\msa,\pa\Lam)>0$. Hence, for $\vr$
sufficiently small, one find actually $|w_\vr(x)|\leq C
\exp(-\frac{c_0\de}{2\vr})<\xi$ if $x\not\in\Lam$.
Therefore, $f(x, |w_\vr|)=g(|w_\vr|)$ for all $\vr>0$
small enough, and the proof of the theorems is
thereby completed.
\end{proof}

\medskip

\noindent {\it Acknowledgment.} \
The authors would like to thank the anonymous reviewer for
his/her helpful comments.
The work was supported by the National Science Foundation of China
(NSFC11331010, 11171286) and the Beijing Center for Mathematics and
Information Interdisciplinary Sciences.

\end{document}